\documentclass[onefignum,onetabnum,review]{siamart190516}

\usepackage{amsmath,amssymb,amsfonts,latexsym,stmaryrd}
\usepackage{graphicx,float,epsfig} 
\usepackage{mathrsfs} 
\usepackage{color}
\usepackage{setspace} 
\usepackage{lipsum}
\usepackage{graphicx,float}
\usepackage{color}
\usepackage{epstopdf}
\usepackage{multirow}
\usepackage{svg}
\usepackage{url}

\usepackage{diagbox}
\ifpdf
  \DeclareGraphicsExtensions{.eps,.pdf,.png,.jpg}
\else
  \DeclareGraphicsExtensions{.eps}
\fi

\def\ds{\displaystyle}

\def\O{\Omega}

\def\g{\gamma}

\def\l{\lambda}

\def\E{K}

\renewcommand\sp{\mathop{\mathrm{Sp}}\nolimits}

\newcommand{\n}{\boldsymbol{n}}
\newcommand{\bb}{\boldsymbol}
\usepackage{booktabs}
\usepackage{float}

\newcommand\bu{\boldsymbol{u}}
\newcommand\bv{\boldsymbol{v}}
\newcommand\bw{\boldsymbol{w}}

\newcommand\bn{\boldsymbol{n}}

\newcommand\bg{\boldsymbol{g}}

\def\hdel{\widehat{\delta}}
\def\CM{\mathcal{X}}
\def\CN{\mathcal{Y}}

\newcommand\bT{\boldsymbol{T}}

\newcommand\0{\mathbf{0}}

\def\CT{{\mathcal T}}

\newcommand\btau{\boldsymbol{\tau}}
\newcommand\bPi{\boldsymbol{\Pi}}

\renewcommand\H{\mathrm{H}}

\renewcommand\L{\mathrm{L}}

\renewcommand\O{\Omega}

\renewcommand\div{\mathop{\mathrm{div}}\nolimits}

\renewcommand\sp{\mathop{\mathrm{sp}}\nolimits}

\newcommand\disp{\displaystyle}

\newcommand{\vertiii}[1]{{\left\vert\kern-0.25ex\left\vert\kern-0.25ex\left\vert #1 
    \right\vert\kern-0.25ex\right\vert\kern-0.25ex\right\vert}}

\newsiamremark{remark}{Remark}
\newsiamremark{hypothesis}{Hypothesis}
\crefname{hypothesis}{Hypothesis}{Hypotheses}
\newsiamthm{problem}{Problem}
\newsiamthm{claim}{Claim}

\headers{Conforming VEM  for the  Oseen eigenvalue problem}{D. Amigo, F. Lepe and N. Verma}

\title{A Conforming virtual element approximation  for the  Oseen eigenvalue problem\thanks{Submitted to the editors DATE.
\funding{The first author was partially supported by ANID-Chile through project Anillo of Computational Mathematics for Desalination Processes ACT210087. The third author was partially supported by ANID-Chile through project Anillo of Computational Mathematics for Desalination Processes ACT210087 and ANID-Chile through FONDECYT de postdoctorado 3230302 (Chile).
}}}

\author{Danilo Amigo \thanks{GIMNAP-Departamento de Matem\'atica, Universidad del B\'io - B\'io, Casilla 5-C, Concepci\'on, Chile. \email{danilo.amigo2101@alumnos.ubiobio.cl}.}
\and Felipe Lepe\thanks{GIMNAP-Departamento de Matem\'atica, Universidad del B\'io - B\'io, Casilla 5-C, Concepci\'on, Chile.  \email{flepe@ubiobio.cl}.}
\and Nitesh Verma\thanks{GIMNAP-Departamento de Matem\'atica, Universidad del B\'io - B\'io, Casilla 5-C, Concepci\'on, Chile.  \email{nverma@ubiobio.cl}.}
}

\usepackage{amsopn}

\ifpdf
\hypersetup{
  pdftitle={Virtual element approximation  for the  Oseen eigenvalue problem},
  pdfauthor={Danilo Amigo  and Felipe Lepe}
}
\fi

\def\CT{{\mathcal T}}

\newcommand\bV{\boldsymbol{V}}
\newcommand\bt{\boldsymbol{t}}

\def\E{K}
\def\bp{\mathbf{p}}
\newcommand\Q{\mathrm{Q}}
\begin{document}

\nolinenumbers
\maketitle

\begin{abstract}
In this paper we analyze a conforming virtual element method to approximate the eigenfunctions and eigenvalues of the two dimensional Oseen eigenvalue problem. We consider the classic velocity-pressure formulation which allows us to consider the divergence-conforming
virtual element spaces employed for the Stokes equations. Under standard assumptions on the meshes we derive a priori error estimates for the proposed method with the aid of the compact operators theory. We report some numerical tests to confirm the theoretical results.

\end{abstract}



\section{Introduction}
\label{sec:intro}
This paper is focused on the numerical approximation of the eigenvalues and eigenfunctions of the Oseen spectral problem. Let  $\O\subset\mathbb{R}^2$ be an open and bounded polygonal domain with Lipschitz boundary $\partial\O$. 
The  equations of the Oseen eigenvalue problem are  given as follows:
\begin{equation}\label{def:oseen-eigenvalue}
\left\{
\begin{array}{rcll}
-\nu\Delta \bu + (\boldsymbol{\beta}\cdot\nabla)\bu + \nabla p&=&\lambda\bu,&\text{in}\,\O,\\
\div \bu&=&0,&\text{in}\,\O,\\
\bu &=&\boldsymbol{0},&\text{on}\,\partial\O,\\
\displaystyle\int_{\O} p &=&0, &\text{in}\,\O,
\end{array}
\right.
\end{equation}
where $\bu$ is the displacement, $p$ is the pressure and $\boldsymbol{\beta}$ is a given vector field, representing a \textit{steady flow velocity}  and $\nu>0$ is the kinematic viscosity.

The Oseen problem is a linearised version of the Navier-Stokes equations. The difference is precisely on the convective term that in \eqref{def:oseen-eigenvalue} is linear, something that on the Navier-Stokes system is not, leading to a number of well known difficulties. The literature on fluid mechanics and more precisely on the Navier-Stokes and related problems as the Oseen system is well established and we refer to \cite{MR851383, John2016} as main references on the mathematical and numerical analysis. Let us mention that these references, among others, are focused on the load problem associated to fluid mechanics.  Our contribution goes beyond that.

Our aim is to analyze numerically the eigenvalue problem \eqref{def:oseen-eigenvalue}. This problem has the particularity of being non-selfadjoint, where the desirable symmetry is no longer valid due the presence of the convective term.  Recently in \cite{MR4728079},  problem \eqref{def:oseen-eigenvalue} has been studied under the approach of a finite element method. Moreover, important properties of the problem at the continuous level are now available, such  as the relation between the spectrum of the Oseen and Stokes eigensystems when the convective term of Oseen vanishes, stability of the Oseen eigenvalue problem and its dual counterpart and, from the numerical point of view,  a priori and a posteriori error estimates and of course, numerical results that help to valid new methods.

The present paper is focused in the development and analysis of a virtual element method (VEM) to approximate the solutions of \eqref{def:oseen-eigenvalue} on the polygonal meshes. To do this task, we need to resort to suitable virtual spaces to approximate the solutions of Stokes and  Navier-Stokes equations, or variations of these systems. On this subject, the literature available is ongoing progress for different formulations such as \cite{MR3164557,MR4082227, MR3626409, MR4015151, MR3614887, MR4332146} among others. However, our intention is to apply these spaces for an eigenvalue problem, which is a different topic that has its intrinsic difficulties in the analysis, which are documented on other works using the VEM to approximate eigenvalues and eigenfunction. On this subject we refer to \cite{adak2022vem, Lepe2021} as recent contributions. Following this, our approach for the analysis of \eqref{def:oseen-eigenvalue} is the classic velocity-pressure formulation in order to extend the results provided by \cite{adak2022vem} for the a priori error analysis of the Stokes eigenvalue problem to our Oseen system, where the conforming virtual spaces under consideration are those proposed in \cite{MR4332146}. The main difference is the lack of symmetry of the Oseen spectral problem, which demands to consider here the discrete dual problem in order to derive error estimates. Hence, the VEM spaces employed by \cite{MR4332146} need to be capable to discretize properly such a dual problem.
According to \cite{MR4728079}, the solution operator for  the Oseen eigenvalue problem results to be compact
and hence, the classic theory for compact operators of \cite{MR1115235} is applied in order to derive error estimates for the eigenvalues and eigenfunctions, together with the spurious free feature of the method. On this subject, it is important to perform an extra computational analysis, since the VEM on its nature, depends on some stabilization terms in order to scale correctly the discrete bilinear forms with respect to the continuous ones. These stabilizations depend on physical parameters and the geometrical properties of the mesh and the domain, so if it is not correctly scaled, spurious eigenvalues may appear as in, for instance, \cite{adak2022vem,MR4742045,Lepe2021, MR4050542}. Hence, a complete computational study of the method is performed in order to confirm the order of convergence and the manner in which the spurious eigenvalues are precisely avoided according to the theory.

The structure of our paper is as follows: Section \ref{sec:model_problem} is dedicated to present the variational framework in which our work is supported, where sesquilinear forms,  primal and adjoint  eigenvalue problems, continuous solution operators and regularity results are presented. In section \ref{sec:vem} we introduce the virtual element method  for the numerical approximation of the problem \eqref{def:oseen-eigenvalue}.
This implies the presentation of the virtual element spaces, degrees of freedom, discrete seminorms, discrete sesquilinear forms and of course the VEM discretization for the primal and adjoint problems. Finally, in section \ref{sec:numerical-experiments}, we present numerical tests 
to show the computed order of convergence in convex and non-convex domains, and also, observed the influence of stabilization parameters on the computation of the spectrum.
\section{The variational formulation}
\label{sec:model_problem}

We begin by assuming that the viscosity $\nu$ is bounded, more precisely, we assume the existence of two constants 
$\nu^+$ and $\nu^{-}$ such that the    $\nu^{-}< \nu< \nu^{+}$. On the steady flow velocity, we assume that  $\boldsymbol{\beta}\in \L^{\infty}(\O,\mathbb{C})^2$ and the following standard assumptions  (see \cite{John2016}):
\begin{itemize}
\item $\|\boldsymbol{\beta}\|_{\infty,\O}\sim 1$ if $\nu\leq \|\boldsymbol{\beta}\|_{\infty,\O}$,
\item $\nu\sim 1$ if $\|\boldsymbol{\beta}\|_{\infty,\O}<\nu$,
\end{itemize}
where the first point is the case more close to the real applications. Moreover, we assume the existence of a number  $\varepsilon_1>0$ such that $\boldsymbol{\beta}\in\L^{2+\varepsilon_1}(\O,\mathbb{C})^2$. This  leads to  the skew-symmetry of the convective term (see \cite[Remark 5.6]{John2016}) which claims that for all $\bv\in\H_0^1(\O,\mathbb{C})^2$, there holds
\begin{equation}
\label{eq:skew}
\int_{\O}(\boldsymbol{\beta}\cdot\nabla)\bv\cdot\bv=0\quad\forall\bv\in\H_0^1(\O,\mathbb{C})^2.
\end{equation}

Now we introduce the functional spaces and norms for our analysis.  To simplify the presentation of the material, let us define the spaces 
$\mathcal{X}:=\H_0^1(\O,\mathbb{C})^2\times \L_0^2(\O,\mathbb{C})$ and   $\mathcal{Y}:=\H_0^1(\O,\mathbb{C})^2\times\H_0^1(\O,\mathbb{C})^2$. For the space $\mathcal{X}$  we define the norm $\|\cdot\|_{\mathcal{X}}^2:=\|\cdot\|_{1,\O}^2+\|\cdot\|_{0,\O}^2$ whereas for $\mathcal{Y}$ the norm will be $\|(\bv,\bw)\|_{\mathcal{Y}}^2=\|\bv\|_{1,\O}^2+\|\bw\|_{1,\O}^2$, for all $(\bv,\bw)\in\mathcal{Y}$. 

Let us introduce the following sesquilinear forms $a:\mathcal{Y}\rightarrow\mathbb{C}$
and $b:\mathcal{X}\rightarrow\mathbb{C}$ defined by 
\begin{equation*}
a(\bw,\bv):=a^{\nabla}(\bw,\bv) + a^{\boldsymbol{\beta}}(\bw,\bv)\quad
\text{and}
\quad
b(\bv,q):=-\int_{\O}q\,\div\bv,
\end{equation*}
where $a^{\nabla},a^{\boldsymbol{\beta}}:\mathcal{Y}\rightarrow\mathbb{C}$ are two sesquilinear forms  defined by
\begin{equation*}
a^{\nabla}(\bw,\bv):=\int_{\Omega}\nu\nabla\bw :\nabla\bv,\quad
\quad
a^{\boldsymbol{\beta}}(\bw,\bv):=\frac{1}{2}\left(\int_{\Omega}(\boldsymbol{\beta}\cdot\nabla)\bw\cdot\bv - \int_{\Omega}(\boldsymbol{\beta}\cdot\nabla)\bv\cdot\bw\right) ,
\end{equation*}
whereas $c(\bw,\bv) := (\bw,\bv)_{0,\O}$ is the standard inner product in $\L^{2}(\O,\mathbb{C})^2$. Observe that the resulting eigenvalue problem is  non-symmetric due the presence of the  sesquilinear form $a^{\boldsymbol{\beta}}(\cdot,\cdot)$. 

With these sesquilinear forms at hand, we write the following weak formulation for \eqref{def:oseen-eigenvalue}: Find $\lambda\in\mathbb{C}$ and $(\boldsymbol{0},0)\neq(\bu,p)\in \mathcal{X}$ such that 
\begin{equation}\label{def:oseen_system_weak}
	\left\{
	\begin{array}{rcll}
a(\bu,\bv) + b(\bv,p)&=&\lambda c(\bu,\bv)&\forall \bv\in \H_0^1(\O,\mathbb{C})^2,\\
b(\bu,q)&=&0&\forall q\in \L_0^2(\O,\mathbb{C}),
\end{array}
	\right.
\end{equation}
where
$$
\L_0^2(\Omega,\mathbb{C}):=\left\{q\in \L^2(\O,\mathbb{C})\,:\,\int_{\O} q =0 \right\}.
$$

Let us define the kernel $\mathcal{K}$ of $b(\cdot,\cdot)$ as follows
\begin{equation*}
\mathcal{K}:=\{\bv\in\H_0^1(\O,\mathbb{C})^2\,:\,  b(\bv, q)=0\,\,\,\,\forall q\in\L_0^2(\O,\mathbb{C})\}.
\end{equation*}

With this space at hand, it is easy to check with the aid of \eqref{eq:skew} that  $a(\cdot,\cdot)$ is $\mathcal{K}$-coercive. Moreover, bilinear form $b(\cdot,\cdot)$ satisfies the following inf-sup condition
\begin{equation}
\label{ec:inf-sup_cont}
\displaystyle\sup_{\btau\in\H_0^1(\O,\mathbb{C})^2}\frac{b(\btau,q)}{\|\btau_h\|_{1,\O}}\geq\beta\|q\|_{0,\O}\quad\forall q\in\L^2_0(\O,\mathbb{C}).
\end{equation}

Let $\bT$ be the so-called solution operator  which is  defined as follows
\begin{equation}\label{eq:operador_solucion_u}
	\bT:\H_0^1(\O,\mathbb{C})^2\rightarrow \H_0^1(\O,\mathbb{C})^2,\qquad \boldsymbol{f}\mapsto \bT\boldsymbol{f}:=\widehat{\bu}, 
\end{equation}
where the pair  $(\widehat{\bu}, \widehat{p})\in\mathcal{X}$ is the solution of the following well-posed source problem
\begin{equation}\label{def:oseen_system_weak_source}
	\left\{
	\begin{array}{rcll}
a(\widehat{\bu}, \bv)+b(\bv,\widehat{p})&=&c(\boldsymbol{f},\bv)&\forall \bv\in \H_0^1(\O,\mathbb{C})^2,\\
b(\widehat{\bu},q)&=&0&\forall q\in \L_0^2(\O,\mathbb{C}),
\end{array}
	\right.
\end{equation}
implying  that $\bT$ is well defined due to the Babu\v{s}ka-Brezzi theory. Moreover, from \cite[Lemma 5.8]{John2016} we have the following estimate for the velocity
\begin{equation}\label{eq:estimatefuente_velocity}
\ds\|\nabla\widehat{\bu}\|_{0,\O}\leq\frac{C_{pf}}{\nu}\|\boldsymbol{f}\|_{\H^{-1}(\O)},
\end{equation}
and for the pressure we have
\begin{equation}
\label{eq:estimatefuente_pressure}
\ds \|\widehat{p}\|_{0,\O}^2\leq \frac{1}{\beta}\left( \|\boldsymbol{f}\|_{\H^{-1}(\O)}+\nu^{1/2}
\|\nabla\widehat{\bu}\|_{0,\O}\left(\nu^{1/2}+C_{pf}\frac{\|\boldsymbol{\beta}\|_{0,\infty}}{\nu^{1/2}}\right)\right),
\end{equation}
where $C_{pf}>0$ represents the constant of the  Poincar\'e-Friedrichs inequality and $\beta>0$ is the inf-sup constant given in \eqref{ec:inf-sup_cont}.

We observe that $(\lambda , (\bu , p)) \in  \mathbb{C}  \times \mathcal{X}$  solves \eqref{def:oseen_system_weak} if and only if $(\kappa , \bu )$ is an eigenpair of $\bT , i.e., \bT \bu  = \kappa \bu$  with $\kappa  := 1/\lambda$.

Under the assumptions on $\boldsymbol{\beta}$,  we have that $a^{\boldsymbol{\beta}}(\cdot,\cdot)$ is well defined and hence, it is enough to consider the  classic Stokes regularity results (see \cite{MR1600081} for instance) in order to derive  the  following additional regularity result for the solutions of the Oseen system.
\begin{theorem}\label{th:regularidadfuente}
There exists $s>0$ such  that for all $\boldsymbol{f} \in \H_0^1(\O,\mathbb{C})^2$, the solution $(\widehat{\bu},\widehat{p})\in\mathcal{X}$ of problem \eqref{def:oseen_system_weak_source}, satisfies for the velocity $\widehat{\bu}\in  \H^{1+s}(\Omega,\mathbb{C})^2$, for the pressure $\widehat{p}\in \H^s(\Omega,\mathbb{C})$, and
 \begin{equation*}
\|\widehat{\bu}\|_{1+s,\O}+\|\widehat{p}\|_{s,\O}\leq C\|\boldsymbol{f}\|_{1,\O},
 .\end{equation*}
where $C := \dfrac{C_{pf}}{\beta}\max\left\lbrace 1, \dfrac{C_{pf}\|\bb{\beta}\|_{\infty,\O}}{\nu}\right\rbrace$ and $\beta > 0$ is the constant associated to the inf-sup condition \eqref{ec:inf-sup_cont}.
\end{theorem}
\label{thm:reg_primal}

The above  regularity result, together with the compact inclusion $\H^{1+s}(\O,\mathbb{C})^2\hookrightarrow \L^2(\O,\mathbb{C})^2$, allows us to conclude the compactness of $\bT$. Now we have the following spectral characterization for $\bT$.
\begin{lemma}(Spectral Characterization of $\bT$).
The spectrum of $\bT$ is such that $\sp(\bT)=\{0\}\cup\{\kappa_{k}\}_{k\in{N}}$ where $\{\kappa_{k}\}_{k\in\mathbf{N}}$ is a sequence of complex eigenvalues that converge to zero, according to their respective multiplicities. 
\end{lemma}

We conclude this section by redefining the spectral problem \eqref{def:oseen_system_weak} in order to  simplify the notations for the forthcoming analysis. With this in mind, let us introduce the sesquilinear form $A:\mathcal{X}\times\mathcal{X}\rightarrow\mathbb{C}$ defined by 
\begin{equation*}
A((\bu,p);(\bv,q)):= a(\bu,\bv)+ b(\bv,p)- b(\bu,q),\quad\forall (\bv,q)\in\mathcal{X}, 
\end{equation*}
which allows us to rewrite problem \eqref{def:oseen_system_weak} as follows: Find $\lambda\in\mathbb{C}$ and $(\boldsymbol{0},0)\neq(\bu,p)\in\mathcal{X}$ such that
\begin{equation}
\label{eq:eigenA}
A((\bu,p),(\bv,q))=\lambda c(\bu,\bv)_{0,\O}\quad\forall (\bv,q)\in\mathcal{X}.
\end{equation}

From \cite[Lemma 2.4]{MR4728079}, we have that the following inf-sup conditions for $A(\cdot,\cdot)$ holds
\begin{equation}\label{eq_infsup_continua}
\begin{split}
\underset{(\bb{0},0)\neq(\bb{w},r) \in \mathcal{X}}{\inf}\underset{(\bb{0},0)\neq(\bb{v},q) \in \mathcal{X}}{\sup} \dfrac{A((\bb{v},q),(\bb{w},r))}{\|(\bb{v},q)\|\|(\bb{w},r)\|} &= \gamma, \\
\underset{(\bb{0},0)\neq(\bb{v},q) \in \mathcal{X}}{\inf}\underset{(\bb{0},0)\neq(\bb{w},r) \in \mathcal{X}}{\sup} \dfrac{A((\bb{v},q),(\bb{w},r))}{\|(\bb{v},q)\|\|(\bb{w},r)\|} &= \gamma,
\end{split}
\end{equation}
where $\gamma$ is a positive constant uniform with respect to $\nu$. 

Since  problem \eqref{def:oseen_system_weak} is non-selfadjoint, it is necessary to introduce the adjoint eigenvalue problem, which  reads as follows: Find $\lambda\in\mathbb{C}$ and a pair $(\boldsymbol{0},0)\neq(\bu^*,p^*)\in\mathcal{X}$ such that  
\begin{equation}\label{def:oseen_system_weak_dual_eigen}
	\left\{
	\begin{array}{rcll}
a(\bv,\bu^*)-b(p^*,\bv)&=&\overline{\lambda}c(\bv,\bu^*)&\forall \bv\in \H_0^{1}(\O,\mathbb{C})^2,\\
-b(q, \bu^*)&=&0&\forall q\in \L_0^2(\O,\mathbb{C}),
\end{array}
	\right.
\end{equation}
Now we introduce the adjoint solution operator  defined  by
\begin{equation}\label{eq:operador_adjunto_solucion_u}
	\bT^*:\H^{1}(\O,\mathbb{C})^2\rightarrow \H^{1}(\O,\mathbb{C})^2,\qquad \boldsymbol{f}\mapsto \bT^*\boldsymbol{f}:=\widehat{\bu}^*, 
\end{equation} 
where $\widehat{\bu}^*\in\H^{1}(\O,\mathbb{C})^2$ is the adjoint velocity of $\widehat{\bu}$ and solves the following adjoint source  problem: Find  
$(\widehat{\bu}^*, \widehat{p}^*)\in\mathcal{X}$ such that 
\begin{equation}\label{def:oseen_system_weak_dual_source}
	\left\{
	\begin{array}{rcll}
a(\bv,\widehat{\bu}^*)-b(\widehat{p}^*,\bv)&=&c(\bv,\boldsymbol{f})&\forall \bv\in \H_0^{1}(\O,\mathbb{C})^2,\\
-b(q, \widehat{\bu}^*)&=&0&\forall q\in \L_0^2(\O,\mathbb{C}).
\end{array}
	\right.
\end{equation}
Similar to Theorem \ref{th:regularidadfuente}, let us assume that the dual source and eigenvalue problems are such that the following estimate holds.

\begin{theorem}\label{th:regularidadfuente_dual}
There exists $s^*>0$ such that for all $\boldsymbol{f} \in \H_0^1(\O,\mathbb{C})^2$, the solution $(\widehat{\bu}^*,\widehat{p}^*)$ of problem \eqref{def:oseen_system_weak_dual_source}, satisfies $\widehat{\bu}^*\in  \H^{1+s^*}(\Omega,\mathbb{C} )^2$ and $\widehat{p}^*\in \H^{s^*}(\Omega,\mathbb{C} )$, and
 \begin{equation*}
\|\widehat{\bu}^*\|_{1+s^*,\O}+\|\widehat{p}^*\|_{s^*,\O}\leq C^*\|\boldsymbol{f}\|_{1,\O},
\end{equation*}
where $C^*>0$.
\end{theorem}

Now we hace the following spectral characterization for  $\bT^*$.
\begin{lemma}(Spectral Characterization of $\bT^*$).
The spectrum of $\bT^*$ is such that $\sp(\bT^*)=\{0\}\cup\{\kappa_{k}^*\}_{k\in{N}}$ where $\{\kappa_{k}^*\}_{k\in\mathbf{N}}$ is a sequence of complex eigenvalues that converge to zero, according to their respective multiplicities. 
\end{lemma}
It is easy to prove that if $\kappa$ is an eigenvalue of $\bT$ with multiplicity $m$, $\overline{\kappa}$ is an eigenvalue of $\bT^*$ with the same multiplicity $m$.

Let us define the sesquilinear form $\widehat{A}:\mathcal{X}\times\mathcal{X}\rightarrow\mathbb{C}$ by
\begin{equation*}
\widehat{A}((\bv,q),(\bu^*,p^*)):=a(\bv,\bu^*)-b(p^*,\bv)+b(q, \bu^*),
\end{equation*}
which allows us to rewrite the dual eigenvalue problem  \eqref{def:oseen_system_weak_dual_eigen} as follows:  Find $\lambda\in\mathbb{C}$ and the pair $(\boldsymbol{0},0)\neq(\bu^*,p^*)\in\mathcal{X}$ such that  
\begin{equation}
\label{eq:eigen_A_dual}
\widehat{A}((\bv,q),(\bu^*,p^*))=\lambda(\bv,\bu^*)\quad\forall (\bv,q)\in\mathcal{X}.
\end{equation}
Similar to the continuous case, from \cite[Lemma 2.7]{MR4728079} we have that $\widehat{A}(\cdot,\cdot)$ satisfies the following inf-sup conditions
\begin{equation*}
\begin{split}
\underset{(\bb{0},0)\neq(\bb{w},r) \in \mathcal{X}}{\inf}\underset{(\bb{0},0)\neq(\bb{v},q) \in \mathcal{X}}{\sup} \dfrac{\widehat{A}((\bb{v},q),(\bb{w},r))}{\|(\bb{v},q)\|\|(\bb{w},r)\|} &= \gamma^*, \\
\underset{(\bb{0},0)\neq(\bb{v},q) \in \mathcal{X}}{\inf}\underset{(\bb{0},0)\neq(\bb{w},r) \in \mathcal{X}}{\sup} \dfrac{\widehat{A}((\bb{v},q),(\bb{w},r))}{\|(\bb{v},q)\|\|(\bb{w},r)\|} &= \gamma^*,
\end{split}
\end{equation*}
where $\gamma^*$ is a positive constant uniform with respect to $\nu$. 
\begin{remark}
Let us note the following: since we are considering primal and dual problems, and consequently we have primal and dual solutions, according to Theorem \ref{thm:reg_primal} and Theorem \ref{th:regularidadfuente_dual}, the corresponding constants on the estimates of these theorems depend on the nature of the problem and, strictly speaking, are not the same. We need to take this into account in the following analysis because, in order to avoid an excess of notation, some constants will be generic, regardless of whether the problem is primal or dual.
\end{remark}

\section{The virtual  element method}
\label{sec:vem}
The following section is dedicated to the analysis of the virtual element approximation for the eigenproblem
presented in Problem~\ref{def:oseen_system_weak}. To do this task, first we need to introduce the following definitions.
Let $\left\{\CT_h\right\}_{h>0}$ be a sequence of decompositions of $\O$
into polygons $\E$. Besides, we will denote by $\ell$ a generic edge of $\partial K$ and by $h_{\ell}$ its length. The set of all the edges in $\CT_h$ will be denoted by $\mathcal{E}_h$. Let us denote by $h_\E$ the diameter of the element $\E$
and $h$ the maximum of the diameters of all the elements of the mesh,
i.e., $h:=\max_{\E\in\O}h_\E$. 

For the mesh $\CT_h$ we will consider
the following assumptions:
\begin{itemize}
\item \textbf{A1.} There exists a constant $\g>0$ such that, for all meshes
$\CT_h$, each polygon $\E\in\CT_h$ is star-shaped with respect to a ball
of radius greater than or equal to $\g h_{\E}$.
\item \textbf{A2.} The distance between any two vertexes of $\E$ is $\geq Ch_\E$, where $C$ is a positive constant.
\end{itemize}

Sesquilinear forms $a(\cdot,\cdot)$, $b(\cdot,\cdot)$ and $c(\cdot,\cdot)$ can be decomposed into local contributions as follows
\begin{align*}
a(\bw,\bv)&:=\sum_{\E\in\CT_{h}}a^{\nabla,\E}(\bw,\bv)+a^{\boldsymbol{\beta},\E}(\bw,\bv)\qquad \text{for all }\bw, \bv\in \H_0^{1}(\O,\mathbb{C})^2,\\
b(\bv,q)&:=\sum_{\E\in\CT_{h}}b^{\E}(\bv,q) \qquad \text{for all }\bv\in  \H_0^{1}(\O,\mathbb{C})^2 \text{and }q\in \L_0^2(\O,\mathbb{C}),\\
c(\bw,\bv)&:=\sum_{\E\in\CT_{h}}c^{\E}(\bw,\bv) \qquad \text{for all }\bw,\bv\in  \H_0^{1}(\O,\mathbb{C})^2.
\end{align*}
In the same way, we split  elementwise the norms of $\H_0^1(\O,\mathbb{C})^2$ and $\L^2(\O,\mathbb{C})$ by 
\begin{align*}
\label{eq:normas}
\ds \|\bv\|_{1,\O}&:=\left(\sum_{K\in\mathcal{T}_h}\|\bv\|_{1,K}^2\right)^{1/2}\quad\forall\bv\in  \H^{1}(\O,\mathbb{C})^2, \\
\|q\|_{0,\O}&:=\left(\sum_{K\in\mathcal{T}_h}\|q\|_{0,K}^2\right)^{1/2}\quad\forall q\in \L_0^2(\O,\mathbb{C}).
\end{align*}

Further, we introduce the following function space of piecewise $\H^1$ functions
\begin{equation*}
\label{disH1Fn}
\H^1(\Omega,\CT_h):= \{ \bv \in \L^2(\Omega,\mathbb{C})^2~: \bv|_K \in \H^1(K,\mathbb{C})^n \},
\end{equation*} associated with the semi-norm
$\displaystyle|\bv|_{1,h}:=\left( \sum_{\E \in \CT_h} |\bv|_{1,\E}^2 \right)^{1/2}$.

Inspired by  \cite{MR4332146,MR4250631}, we will construct the virtual spaces to analyze the discrete problem.
Let us begin with  the velocity virtual space: 
for a simple polygon  $\E$  we  define 
\begin{equation*}
\label{espace-1}
\boldsymbol{\mathbb{B}}_{\partial \E}:=\{\bv_{h}\in [C^0(\partial \E)]^{2}: \bv_{h}\cdot \bn|_{\ell}\in 
\mathbb{P}_2(\ell)\text{ and }\bv_{h}\cdot \bt|_{\ell}\in 
\mathbb{P}_1(\ell)\ \ \forall \ell\subset
\partial \E\}.
\end{equation*}
With this space at hand, we consider the
following local finite dimensional 
space:
\begin{multline*}
 \label{space2}
\widehat{\boldsymbol{V}}_h^\E:=\left\{\bv_{h}\in [\H^1(\E)]^2: \Delta \bv_{h}-\nabla s\in \nabla \mathbb{P}_2(\E)^{\perp} \text{ for some } s\in \L_0^2(E), \right.\\
\left.\div \bv_h\in \mathbb{P}_0(\E)\,\, 
\textrm{and}\,\,\bv_{h}|_{\partial \E}\in \boldsymbol{\mathbb{B}}_{\partial \E} \right\}.
 \end{multline*}
The following linear operators are well defined for all $\bv_h\in\widehat{\bV}_h^\E$:
 \begin{itemize}
  \item $\mathcal{V}_\E^h$ : The (vector) values of $\bv_h$ at the vertices.
  \item $\mathcal{L}_\E^h$: the value of
  \begin{equation*}
  \frac{1}{|\ell|}\int_{\ell}\bv_h\cdot\bn \quad \forall \text{ edges }\ell\in\partial\E.
  \end{equation*}
  \item $\mathcal{\E}_\E^h$: the value of
  \begin{equation*}
  \int_\E\bv_h\cdot g^{\perp} \quad \forall \bg^{\perp}\in \nabla \mathbb{P}_2(\E)^{\perp}.
  \end{equation*}
\end{itemize}
Let us remark that the set of linear operators $\mathcal{V}_\E^h$,   $\mathcal{L}_\E^h$, and $\mathcal{\E}_\E^h$ constitutes a set of degrees of
freedom for the local virtual space $\widehat{\bV}_h^\E$ (see \cite{MR4332146}). Moreover, it is easy to check that
$[\mathbb{P}_1(\E)]^2\subset \widehat{\bV}_h^\E$. This will guarantee the good approximation properties for the space. Now we define the projector
$\bPi^{\nabla,\E}: \widehat{\bV}_h^\E\longrightarrow [\mathbb{P 
}_1(\E)]^2\subset\widehat{\bV}_h^\E$
for each $\bv_h\in\widehat{\bV}_h^\E$ as the solution of
\begin{align*}
\label{proje_0}
\left\{\begin{array}{ll}
  a^{\nabla,\E}(\bp,\boldsymbol{\Pi}^{\nabla,\E}\bv_h) &= a^{\nabla,\E}(\bp,\bv_h)
\quad \forall \bp\in [\mathbb{P}_1(\E)]^2,\\\\
\ds\vert\partial\E\vert^{-1}\int_{\partial \E}\boldsymbol{\Pi}^{\nabla,\E}\bv_h
&=\ds\vert\partial\E\vert^{-1}\int_{\partial \E}\bv_h.
\end{array}\right.
\end{align*}

We observe that $\boldsymbol{\Pi}^{\E}$ is well defined on $\widehat{\bV}_h^\E$ and computable. Now we define the local virtual element space $\bV_h^\E$ as follows,
\begin{equation*}
 \label{space23}
\bV_h^\E:=\left\{\bv_{h}\in \widehat{\bV}_h^\E:  \disp\int_\E \boldsymbol{\Pi}^{\nabla,\E}\bv_h\cdot \bg^{\perp}=\int_\E \bv_h\cdot \bg^{\perp} \text{ for } \bg^{\perp}\in \nabla \mathbb{P}_2(\E)^{\perp}\right\}.
 \end{equation*}
In addition,  the  standard $[{\mathrm L}^2(\E,\mathbb{C})]^2$-projector operator
$\boldsymbol{\Pi}_{0}^{\E}: \bV_h^{\E}\to[\mathbb{P}_1(\E)]^2$  can be  computed. In fact, for all $\bv_h\in  \bV_h^\E$, the function $\boldsymbol{\Pi}_{0}^{\E}\bv_h\in [\mathbb{P}_1(\E)]^2$ is defined by:
\begin{align*}
\label{pi0}
\int_\E\boldsymbol{\Pi}_{0}^{\E}\bv_h\cdot \boldsymbol{p}&=\int_\E\bv_h\cdot \left(\nabla p_2+\nabla p_2^{\perp}\right)\\
&=-\int_\E\div(\bv_h)p_2+\int_{\partial \E}(\bv_h\cdot\n)p_2+
 \disp\int_\E \boldsymbol{\Pi}^{\nabla,\E}\bv_h\cdot\nabla p_2^{\perp},
\end{align*}
for every $\boldsymbol{p}\in[\mathbb{P}_1(\E)]^2$.
\begin{remark}
In the case when we need to denote the aforementioned projectors as global, we will drop the superindex $\E$ when is necessary.
\end{remark}

 Now we introduce  the global virtual space:
for every decomposition $\CT_h$ of $\O$ into
simple polygons $\E$, we define
\begin{align*}
\bV_h:=\left\{\bv_{h}\in\H_0^1(\O,\mathbb{C})^2:\bv_h|_{\E}\in  \bV_h^{\E},\quad \forall \E\in \CT_h \right\}.
\end{align*}
In agreement with the local choice of the degrees of freedom, in 
$\bV_h$ we choose the following degrees of freedom:
\begin{itemize}
\item $\mathcal{V}^h$: the (vector) values of $\bv_h$ at the vertices of $\CT_h$,
\item $\mathcal{L}^h$: the value of
  \begin{equation*}
  \frac{1}{|\ell|} \int_\ell\bv_h\cdot\bn \quad \forall \ell\in\CT_h.
  \end{equation*}
\end{itemize}
On the other hand, the pressure space is given by
\begin{equation*}
\Q_h:=\{q_h \in \L^{2}(\O,\mathbb{C})\,:\, q_h|_K\in\mathbb{P}_0(K),\,\quad\forall K\in \CT_h\},
\end{equation*}
where the degrees of freedom are one per element, given by the value of the function on the element. Moreover and for simplicity,  we define $\mathcal{X}_h:=\bV_h\times\Q_h$.

\subsection{The discrete sesquilinear  forms}
\label{subsec:disc_b_f}
Now we will introduce the discrete versions of the sesquilinear forms $a(\cdot, \cdot)$, $b(\cdot,\cdot)$ and $c(\cdot, \cdot)$. Let us  begin by introducing the discrete counterpart of $a^{\nabla}(\cdot,\cdot)$, which is defined by
\begin{equation*}
\label{ec:ahch}
a^{\nabla}_h(\bw_h,\bv_h)
:=\sum_{\E\in\CT_h}a_h^{\nabla,\E}(\bw_h,\bv_h),
\end{equation*}
where $a_h^{\nabla,\E}(\cdot,\cdot)$ is the sesquilinear form defined on
$\bV_{h}^{\E}\times\bV_{h}^{\E}$ by
\begin{equation*}
\label{21}
a_h^{\nabla,\E}(\bw_{h},\bv_{h})
:=a^{\nabla,\E}\big(\boldsymbol{\Pi}^{\nabla,\E} \bw_h,\boldsymbol{\Pi}^{\nabla,\E} \bv_h\big)
+S^{\E}\big(\bw_h-\bPi^{\nabla,\E} \bw_h,\bv_h-\bPi^{\nabla,\E} \bv_h\big),
\end{equation*}
for every $\bw_h,\bv_h\in\bV_{h}^{\E}$, where $S^{\E}(\cdot,\cdot)$ denotes any symmetric positive definite bilinear form that satisfies
$$c_0 a^{\nabla,\E}(\bw_h,\bw_h)\leq S^{\E}(\bw_h,\bw_h)\leq c_1a^{\nabla,\E}(\bw_h,\bw_h)\qquad \forall \bw_h\in\bV_{h}^{\E} \cap \text{ker}(\bPi^{\nabla,\E}),$$ 
with $c_0$ and $c_1$ being positive constants depending on the mesh assumptions. 
Moreover, it is easy to check that $a_h^{\nabla,\E}(\cdot,\cdot)$ must admit the following properties:
\
\begin{itemize}
\item Consistency: For all $\boldsymbol{q}_h\in[\mathbb{P}_1(\E)]^2$ and $\bv_h\in\mathbf{V}_h^\E$
\begin{equation*}
a_h^{\nabla,\E}(\boldsymbol{q}_h,\bv_h)=a^{\nabla,\E}(\boldsymbol{q}_h,\bv_h),
\end{equation*}
\item Stability: There exists two positive constants $\alpha_*$ and $\alpha^*$, independent of $h$ and $\E$ such that
\begin{equation*}
\alpha_*a^{\nabla,\E}(\bv_h,\bv_h)\leq a_h^{\nabla,\E}(\bv_h,\bv_h)\leq\alpha^*a^{\nabla,\E}(\bv_h,\bv_h)\quad\forall\bv_h\in\mathbf{V}_h^\E.
\end{equation*}
\end{itemize}
Also, we introduce the discrete counterpart of $a^{\boldsymbol{\beta},K}(\cdot,\cdot)$, which is defined by 
\begin{equation}\label{eq:convecdisc}
\widehat{a}_{h}^{\bb{\beta},\E}(\bw_h,\bv_h) := \frac{1}{2}\left(a_{h}^{\boldsymbol{\beta},\E}(\bw_h,\bv_h) - a_{h}^{\boldsymbol{\beta},\E}(\bv_h,\bw_h)\right) \quad \forall \bw_h,\bv_h \in \bV_h^\E,
\end{equation}
where $a_h^{\bb{\beta},\E}(\bw_h,\bv_h):=((\Pi_{0}^{0,\E}\nabla\bw_h)\bb{\beta},\bb{\Pi}_0^\E\bv_h)_{0,\E}$ and $\Pi_{0}^{0,\E} : \nabla \bV_h^\E \rightarrow \mathbb{P}_0(\E)^{2\times 2}$ is the $\L^2$-orthogonal projection onto $\mathbb{P}_0(\E)^{2\times 2}$. Then, we define 
\begin{equation*}
\widehat{a}^{\bb{\beta}}_h(\bw_h,\bv_h)
:=\sum_{\E\in\CT_h}\widehat{a}_h^{\bb{\beta},\E}(\bw_h,\bv_h),
\end{equation*}

Now, we define  $c_h^{\E}(\cdot,\cdot)$ as the sesquilinear form defined on
$\bV_{h}^{\E}\times\bV_{h}^{\E}$ by
\begin{equation*}
\label{eq_ch}
c_h^{\E}(\bw_h,\bv_h)
:=c^{\E}\big(\bPi_{0}^{\E} \bw_h,\bPi_{0}^{\E} \bv_h\big)
\qquad \forall \bw_h,\bv_h\in\bV_{h}^{\E}. 
\end{equation*}

According to the definition of  $c_h^\E(\cdot,\cdot)$ and the approximation properties of $\bPi_0^\E$, it is not difficult to prove that:
\begin{equation*}
c_h^\E(\boldsymbol{q}_h,\bv_h)=c^\E(\boldsymbol{q}_h,\bv_h),\qquad \boldsymbol{q}_h\in [\mathbb{P}_1(\E)]^2,
\end{equation*}
and  $\|\bv-\bPi_0^\E\bv\|_{0,\E}=\inf_{\boldsymbol{q}_h\in [\mathbb{P}_1(\E)]^2}\|\bv-\boldsymbol{q}\|_{0,\E}$.
We remark that for all $\E\in\CT_h$, the local sesquilinear form $b^{\E}(\cdot,\cdot)$ is computable from the degrees of freedom. In fact, for any function $\bv_h\in \bV_h^K$  and $q_{h}\in\Q_{h}$ we have
\begin{equation*}
b^{\E}(\bv_h, q_h)=\int_\E\div\bv_hq_h=\sum_{\ell\subset\partial \E}\int_\ell(q_{h}\bn)\cdot\bv_h.
\end{equation*}
Finally, the global discrete counterparts of $a(\cdot,\cdot)$ and $c(\cdot,\cdot)$ are defined by
\begin{equation*}
a_h(\bw_h,\bv_h)
:= \widehat{a}_h^{\bb{\beta}}(\bw_h,\bv_h) + a_h^{\nabla}(\bw_h,\bv_h), \quad c_h(\bw_h,\bv_h)
:=\sum_{\E\in\CT_h}c_h^{\E}(\bw_h,\bv_h).
\end{equation*} 
With all these ingredients at hand, the discrete version of \eqref{def:oseen_system_weak} reads as follows: Find $\lambda_h\in\mathbb{C}$ and $(\boldsymbol{0},0)\neq (\bu_h,p_h)\in \mathcal{X}_h$ such that
\begin{equation}\label{def:oseen_system_weak_disc}
	\left\{
\begin{array}{rcll}
 a_h(\bu_h,\bv_h)+b(\bv_h,p_h) & = &\lambda_h c_{h}(\boldsymbol{u}_h,\bv_h) &\forall\bv_h\in \mathbf{V}_h,\\
\ds b(\bu_h,q_h) & = & 0 &\forall\ q_h\in \Q_h.
\end{array}
\right.
\end{equation}

From \eqref{def:oseen_system_weak_disc} we have for the sesquilinear form $b(\cdot,\cdot)$, the following inf-sup condition given by \cite[Lemma 4.3]{MR4332146}. More precisely, 
there exists a positive constant $\bar{\beta}$, independent of $h$, such that 
\begin{equation}
\label{eq:inf_sup_disc}
\ds\sup_{\bb{0}\neq\bv_h\in \bV_h}\frac{b(\bv_h,q_h)}{\|\bv_h\|_{1,\O}}\geq\bar{\beta}\|q_h\|_{0,\O}\qquad\forall q_h\in \Q_h.
\end{equation}

On the other hand, observe that \eqref{eq:convecdisc} vanishes when $\bw_h=\bv_h$, obtaining easily the coercivity of $a_h(\cdot,\cdot)$, which  depends only of $a_h^{\nabla}(\cdot,\cdot)$. Then, following \cite[Section 3.2]{MR4332146}, there exists $\tilde{\alpha}>0$, independent of $h$ such that 
\begin{equation}
\label{eq:elliptic_disc}
a_h(\bv_h,\bv_h)\geq \tilde{\alpha} \|\bv_h\|_{1,\O}^2\quad\forall\,\bv_h\in \bV_h.
\end{equation}
%

Let us introduce the discrete solution operator $\bT_h$ defined  by 
\begin{align*}
\bT_h:\H_0^1(\O,\mathbb{C})^2\rightarrow \bV_{h},\quad
           \boldsymbol{f}\mapsto \bT_h\boldsymbol{f}:=\widehat{\bu}_h, 
\end{align*}
where $\widehat{\bu}_h$ is the solution of the following discrete source problem: Given a source $\boldsymbol{f}\in \H_0^1(\O,\mathbb{C})^2$, find
$\widehat{\bu}_h\in \bV_h$ such that
\begin{equation}\label{eq:weak_oseen_source_disc}
\left\{
\begin{array}{rcll}
 a_h(\widehat{\bu}_h,\bv_h)+b(\bv_h,\widehat{p}_h) & = &c_{h}(\boldsymbol{f},\bv_h) &\forall\bv_h\in \bV_{h},\\
\ds b(\widehat{\bu}_h,q_h) & = & 0 &\forall\ q_h\in \Q_h.
\end{array}
\right.
\end{equation}

From \eqref{eq:inf_sup_disc} and \eqref{eq:elliptic_disc} and  the Babu\v ska-Brezzi theory, we have that $\bT_h$ is well defined.
Moreover, as in the continuous case,  it is easy to check that $(\kappa_h,\bu_h)$ is an eigenpair of $\bT_h$ if only if there exists $(\bu_h,p_h)\in\mathcal{X}_h$ such that $(\lambda_h, (\bu_h,p_h))\in\mathbb{C}\times\mathcal{X}_h$ solves \eqref{def:oseen_system_weak_disc}, i.e. $\ds \bT_h\bu_h=\mu_h\bu_h$ with $\mu_h:=1/\lambda_h$ and $\l_h\neq0$.

Analogously to the continuous case, we rewrite problem \eqref{eq:weak_oseen_source_disc} as follows: Given $\boldsymbol{f}\in[\L^2(\O)]^2$, find  $(\bu_h,p_h)\in \mathcal{X}_h$ such that
\begin{equation*}
\label{eq:formulation_A_h}
A_h((\bu_h,p_h),(\bv_h,q_h))=c_{h}(\boldsymbol{f},\bv_h),\quad (\bv_h,q_h)\in \bV_{h}\times \Q_h,
\end{equation*}
where $A_h: \mathcal{X}_h\times \mathcal{X}_h\rightarrow\mathbb{C}$ is the sesquilinear form defined by
 \begin{equation*}
 \label{ec:forma_Ah}
A_h((\bu_h,p_h),(\bv_h,q_h)):= a_h(\bu_h,\bv_h)+b(\bv_h,p_h)- b(\bu_h,q_h),
\end{equation*}
for all  $(\bu_h,p_h), (\bv_h,q_h)\in \mathcal{X}_h$.


\subsection{Technical approximation results}
The following approximation result for
polynomials in star-shaped domains (see for instance \cite{MR2373954}),
 is derived from results of interpolation between Sobolev spaces
(see for instance \cite[Theorem~I.1.4]{MR851383}) leading to an analogous
result for integer values of $s$. 

\begin{lemma}[Existence of a virtual approximation operator]
\label{lmm:bh}
If assumption {\bf A1} is satisfied, then  for every $s$ with 
$0\le s\le 1$ and for every $\bv\in[\H^{1+s}(K,\mathbb{C})]^2$, there exists
$\bv_{\pi}\in\mathbb{P}_1(\E)$ such that
$$
\left\|\bv-\bv_{\pi}\right\|_{0,\E}
+h_{\E}\left|\bv-\bv_{\pi}\right|_{1,\E}
\lesssim  h_{\E}^{1+s}\left\|\bv\right\|_{1+s,\E},
$$
where the hidden constant  depends only on mesh regularity constant  $\gamma$.
\end{lemma}


\noindent We also have the following result that provides the existence of an interpolated function for the velocity, together with an approximation error (see \cite[Lemma  4.2]{MR4332146}).
\begin{lemma}[Existence of an interpolation operator]
\label{estimate2}
Under the assumptions  {\bf A1} and  {\bf A2}, let $\bv\in \H_0^1(\O,\mathbb{C})^2\cap[\H^{1+s}(\O,\mathbb{C})]^2$, with $0\leq s\leq 1$. Then, there exists $\bv_I\in \bV_h$ such that
\begin{equation*}
\|\bv-\bv_I\|_{0,K}+h_K|\bv-\bv_I|_{1,K}\lesssim h_K^{1+s}|\bv|_{1+s,K},
\end{equation*}
where the hidden constant is  positive and independent of $h_K$.
\end{lemma}
\noindent Finally, let $\mathcal{P}_{h}:\L^2(\O,\mathbb{C})\to \left\{q\in \L^2(\O,\mathbb{C}): q|_{\E}\in\mathbb{P}_{0}(\E)\right\}$  be the $\L^2(\O,\mathbb{C})$-orthogonal projector which satisfies the following approximation property.
\begin{lemma}
\label{estimate3}
If $0\leq s\leq 1$, it holds
\begin{equation*}
\|q-\mathcal{P}_{h}(q)\|_{0,\O}\lesssim h^{s}\|q\|_{s,\O} \qquad \forall q\in\H^s(\O,\mathbb{C})\cap \Q.
\end{equation*}
\end{lemma}

As in the continuous case, we also have an inf-sup condition for $A_h(\cdot,\cdot)$. This is stated in the following result.

\begin{lemma}\label{eq:infsupdiscreta}
The following inf-sup conditions for $A_h(\cdot,\cdot)$ holds
\begin{equation*}
\begin{split}
\underset{(\bb{0},0)\neq(\bb{w}_h,r_h) \in \mathcal{X}_h}{\inf}\underset{(\bb{0},0)\neq(\bb{v}_h,q_h) \in \mathcal{X}_h}{\sup} \dfrac{A_h((\bb{v}_h,q_h),(\bb{w}_h,r_h))}{\|(\bb{v}_h,q_h)\|\|(\bb{w}_h,r_h)\|} &= \widetilde{\gamma}, \\
\underset{(\bb{0},0)\neq(\bb{v}_h,q_h) \in \mathcal{X}_h}{\inf}\underset{(\bb{0},0)\neq(\bb{w}_h,r_h) \in \mathcal{X}_h}{\sup} \dfrac{A_h((\bb{v}_h,q_h),(\bb{w}_h,r_h))}{\|(\bb{v}_h,q_h)\|\|(\bb{w}_h,r_h)\|} &= \widetilde{\gamma}.
\end{split}
\end{equation*}
\end{lemma}
\begin{proof}
The proof is divided in three steps. 

\noindent \textbf{Step 1.} Observe that
\begin{equation*}
A_h((\bb{v}_h,q_h),(\bb{v}_h,q_h)) = a_h(\bb{v}_h,\bb{v}_h) \geq C\|\bb{v}_h\|_{1,\O}^{2},
\end{equation*}
and invoking the Archimedean property, there exists $\varepsilon \in (0,1]$ such that
\begin{equation*}
A_h((\bb{v}_h,q_h),(\bb{v}_h,q_h)) = a_h(\bb{v}_h,\bb{v}_h) \geq C\|\bb{v}_h\|_{1,\O}^{2} \geq C\varepsilon\|(\bb{v}_h,q_h)\|^{2}.
\end{equation*}
\textbf{Step 2.} For $\bb{w}_h \in \bV_h$ we have
\begin{equation*}
A_h((\bb{v}_h,q_h),(\bb{w}_h,0)) = a_h(\bb{v}_h,\bb{w}_h) + b(\bb{w}_h,q_h).
\end{equation*}
Invoking the inf-sup condition for $b(\cdot,\cdot)$, for $q_h \in Q_h$ there exists $\widetilde{\bb{w}}_h \in \bV_h$ such that 
\begin{equation*}
b(\widetilde{\bb{w}}_h,q_h) \geq \widetilde{\beta}\|\widetilde{\bb{w}}_h\|_{1,\O}\|q_h\|_{0,\O},
\end{equation*}
and defining  $\bb{w}_h := \dfrac{\|q_h\|_{0,\O}}{\|\widetilde{\bb{w}}_h\|_{1,\O}}\widetilde{\bb{w}}_h$ we have
\begin{equation*}
\|\bb{w}_h\|_{1,\O} = \|q_h\|_{0,\O} \quad \text{and} \quad b(\bb{w}_h,q_h) \geq \widetilde{\beta}\|q_h\|_{0,\O}^{2}\quad\forall q_h\in Q_h.
\end{equation*}
Therefore, we obtain $-b(\bb{w}_h,q_h) \leq -\widetilde{\beta}\|q_h\|_{0,\O}^2$. 

On the other hand it is direct the following estimate
\begin{equation*}
a_h(\bb{v}_h,\bb{w}_h) \leq C\max\{\|\bb{\beta}\|_{\infty,\O},\nu\}\|\bb{v}_h\|_{1,\O}\|\bb{w}_h\|_{1,\O} \leq C_{\bb{\beta},\nu}\|(\bb{v}_h,q_h)\|^2.
\end{equation*}
Hence we obtain
\begin{equation*}
A_h((\bb{v}_h,q_h),(\bb{w}_h,0)) \leq -\widetilde{\beta}\|q_h\|_{0,\O}^2 + C_{\bb{\beta},\nu}\|(\bb{v}_h,q_h)\|^2.
\end{equation*}
\textbf{Step 3.} Substituting $\bb{w}_h$ by $-\bb{w}_h$ in the last estimate on \textbf{Step 2}, we obtain 
\begin{equation}
A_h((\bb{v}_h,q_h),(-\bb{w}_h,0)) = -A_h((\bb{v}_h,q_h),(\bb{w}_h,0)) \geq \widetilde{\beta}\|q_h\|_{0,\O}^2 - C_{\bb{\beta},\nu}\|(\bb{v}_h,q_h)\|^{2}.
\end{equation}
With the above estimate at hand and invoking \cite[Lemma 5]{BRAACK20111126}, we conclude directly that
\begin{equation*}
A_h((\bb{v}_h,q_h),(\bb{w}_h,r_h)) \geq \widetilde{\gamma}\|(\bb{v}_h,q_h)\|\|(\bb{w}_h,r_h)\|.
\end{equation*}
This concludes the proof.
\end{proof}
As a consequence of Lemma \ref{eq:infsupdiscreta}, we have that $A_h(\cdot,\cdot)$ is stable in the sense that given $(\bb{v}_h,q_h) \in \mathcal{X}_h$, there exists a pair $(\bb{w}_h,s_h) \in \mathcal{X}_h$ such that $\|\bb{v}_h\|_{1,\O} + \|q_h\|_{0,\O} \leq A_h((\bb{v}_h,q_h),(\bb{w}_h,s_h))$ and $\|\bb{w}_h\|_{1,\O}+\|s_h\|_{0,\O} \leq C$.

\subsection{The dual discrete eigenvalue problem}
With our discrete spaces at hand, we are in position to introduce the virtual element discretization of \eqref{def:oseen_system_weak_dual_eigen} which reads as follows:  Find $\lambda_h\in\mathbb{C}$ and $(\boldsymbol{0},0)\neq(\bu_h^*,p_h^*)\in \mathcal{X}_h$ such that 
\begin{equation*}\label{def:stokes_system_weak_disc}
	\left\{
	\begin{array}{rcll}
a_h(\bv_h,\bu_h^*) - b(\bv_h,p_h^*)&=&\overline{\lambda_h}c_h(\bv_h,\bu_h^*)&\forall \bv_h\in\boldsymbol{V}_h,\\
-b(\bu_h^*,q_h)&=&0&\forall q_h\in\Q_h.
\end{array}
	\right.
\end{equation*}
Then, we introduce the adjoint of $\bT_h$, which is defined as follows
\begin{equation*}\label{eq:operador_solucion_dual_u_h}
	\bT_h^*: \H^{1}(\O,\mathbb{C})^2 \rightarrow \boldsymbol{V}_h, \qquad \boldsymbol{f}^*\mapsto \bT_h^*\boldsymbol{f}^*:=\widehat{\bu}_h^*, 
\end{equation*}
where the pair  $(\widehat{\bu}_h^*, \widehat{p}_h^*)\in \mathcal{X}_h$ is the solution of the following well posed dual discrete source problem
\begin{equation}\label{def:oseen_system_weak_disc_source_dual}
	\left\{
	\begin{array}{rcll}
a_h(\bv_h^*,\widehat{\bu}_h^*) - b(\bv_h,\widehat{p}_h^*)&=& c_h(\bv_h,\boldsymbol{f}^*)&\forall \bv_h\in\boldsymbol{V}_h,\\
-b(\widehat{\bu}_h^*,q_h)&=&0&\forall q_h\in\Q_h.
\end{array}
	\right.
\end{equation}
implying  that $\bT_h^*$ is well defined due to the Babu\^{s}ka-Brezzi theory. 

The discrete counterpart of \eqref{eq:eigenA} is defined by
\begin{equation*}
\widehat{A}_{h}((\bb{v}_h,q_h),(\bb{u}_h^*,p_h^*)):=a_{h}(\bb{v}_h,\bb{u}_h^*)-b(p_h,\bb{v})+b(\bb{u}_h^*,q_h).
\end{equation*}

As in the primal problem, we also have an inf-sup condition for $A_h(\cdot,\cdot)$. The proof is similar to those in Lemma \ref{eq:infsupdiscreta}, so we skip the details.

\begin{lemma}\label{eq:infsupdiscretadual}
The following inf-sup conditions for $\widehat{A}_h(\cdot,\cdot)$ holds
\begin{equation*}
\begin{split}
\underset{(\bb{0},0)\neq(\bb{w}_h,r_h) \in \mathcal{X}_h}{\inf}\underset{(\bb{0},0)\neq(\bb{v}_h,q_h) \in \mathcal{X}_h}{\sup} \dfrac{\widehat{A}_h((\bb{v}_h,q_h),(\bb{w}_h,r_h))}{\|(\bb{v}_h,q_h)\|\|(\bb{w}_h,r_h)\|} &= \widetilde{\gamma}^*, \\
\underset{(\bb{0},0)\neq(\bb{v}_h,q_h) \in \mathcal{X}_h}{\inf}\underset{(\bb{0},0)\neq(\bb{w}_h,r_h) \in \mathcal{X}_h}{\sup} \dfrac{\widehat{A}_h((\bb{v}_h,q_h),(\bb{w}_h,r_h))}{\|(\bb{v}_h,q_h)\|\|(\bb{w}_h,r_h)\|} &= \widetilde{\gamma}^*.
\end{split}
\end{equation*}
\end{lemma}

Finally, we observe that the discrete eigenvalues associated to $\bT_h^*$ are the conjugates of the eigenvalues of $\bT_h$.


Now, due to the compactness of $\bT$, we are able to prove that $\bT_h$ converge to $\bT$ as $h\downarrow 0$  in norm. This is contained in the following result.
\begin{lemma}
\label{lmm:conv1}
Let $\boldsymbol{f}\in \H^1_0(\Omega,\mathbb{C})$ be such that $\widehat{\bu}:=\boldsymbol{T}\boldsymbol{f}$ and $\widehat{\bu}_h:=\bT_h\boldsymbol{f}$. Then, there exists a positive constant $\bb{C}$, independent of $h$, such that 
\begin{equation*}
\|(\bT-\bT_h)\boldsymbol{f}\|_{1,\Omega}\leq  \bb{C}h^{\sigma}\|\boldsymbol{f}\|_{1,\O}, \quad \sigma := \min\{1,s\}.
\end{equation*}
\end{lemma}
\begin{proof}
Let $\boldsymbol{f}\in\H_0^1(\O)$ be such that $\widehat{\bu}:=\bT\boldsymbol{f}$ and $\widehat{\bu}_h:=\bT_h\boldsymbol{f}$, where the pairs $(\widehat{\bu},\widehat{p})\in\mathcal{X}$  and  $(\widehat{\bu}_h,\widehat{p}_h)\in\mathcal{X}_h$ are the solutions of \eqref{def:oseen_system_weak_source} and \eqref{eq:weak_oseen_source_disc}, respectively. Then, we have
\begin{multline}
\label{eq:danilo}
\|(\bT-\bT_h)\boldsymbol{f}\|_{1,\O} = \|\widehat{\bu}-\widehat{\bu}_h\|_{1,\O} \leq \|\widehat{\bu}-\widehat{\bu}_h\|_{1,\O} + \|\widehat{p}-\widehat{p}_h\|_{0,\O} \\
\leq \|\widehat{\bu}-\widehat{\bu}_I\|_{1,\O} + \|\widehat{p}-\mathcal{P}_h(\widehat{p})\|_{0,\O} + \|\widehat{\bu}_I-\widehat{\bu}_h\|_{1,\O} + \|\widehat{p}_h-\mathcal{P}_{h}(\widehat{p})\|_{0,\O}.
\end{multline}
Observe that invoking Lemmas \ref{estimate2} and \ref{estimate3} we obtain
\begin{equation*}
\|\widehat{\bu} - \widehat{\bu}_I\|_{1,\O} \lesssim h^{s}|\widehat{\bu}|_{1+s,\O} \quad \text{and} \quad \|\widehat{p} - \mathcal{P}_h(\widehat{p})\|_{0,\O} \lesssim h^{s}\|\widehat{p}\|_{s,\O}.
\end{equation*}
Now, we need to control the last two terms on the right-hand side of \eqref{eq:danilo}. First we note that, since $A_h(\cdot,\cdot)$ is stable we have
\begin{equation}
\label{eq:danilo2}
\|\widehat{\bu}_I-\widehat{\bu}_h\|_{1,\O} + \|\widehat{p}_h-\mathcal{P}_{h}(\widehat{p})\|_{0,\O} \leq A_{h}((\widehat{\bu}_I-\widehat{\bu}_h,\widehat{p}_h-\mathcal{P}_{h}(\widehat{p})),(\bv_h,q_h)),
\end{equation}
with $(\bv_h,q_h) \in \mathcal{X}_h$ being chosen in such a way that $\|\bv_h\|_{1,\O} + \|q_h\|_{0,\O} \leq 1$. On the other hand, manipulating the right-hand side of \eqref{eq:danilo2} we obtain 
\begin{multline*}
A_{h}((\widehat{\bu}_I-\widehat{\bu}_h,\widehat{p}_h-\mathcal{P}_{h}(\widehat{p})),(\bv_h,q_h)) = A_{h}((\widehat{\bu}_I,\mathcal{P}_{h}(\widehat{p})),(\bv_h,q_h)) - c_h(\boldsymbol{f},\bv_h) \\
= \displaystyle \sum_{\E\in\CT_h} \left [a_h^{\nabla,\E}(\widehat{\bu}_I,\bv_h) + \widehat{a}_h^{\boldsymbol{\beta},\E}(\widehat{\bu}_I,\bv_h) + b^\E(\bv,\mathcal{P}_h(\widehat{p})) + b^\E(\widehat{\bu}_I,q_h)\right] - c_h(\boldsymbol{f},\bv_h) \\
= \displaystyle \sum_{\E\in\CT_h} \left [a_h^{\nabla,\E}(\widehat{\bu}_I - \widehat{\bu}_\pi,\bv_h) + a^{\nabla,\E}(\widehat{\bu}_\pi - \widehat{\bu},\bv_h) + a^{\nabla,\E}(\widehat{\bu},\bv_h) + \widehat{a}_h^{\boldsymbol{\beta},\E}(\widehat{\bu}_I,\bv_h)\right. \\
\displaystyle \left. + b^\E(\bv,\mathcal{P}_h(\widehat{p})) + b^\E(\widehat{\bu}_I,q_h)\right] - c_h(\boldsymbol{f},\bv_h) \\
= \displaystyle \sum_{\E\in\CT_h} \left [a_h^{\nabla,\E}(\widehat{\bu}_I - \widehat{\bu}_\pi,\bv_h) + a^{\nabla,\E}(\widehat{\bu}_\pi - \widehat{\bu},\bv_h) + \widehat{a}_h^{\boldsymbol{\beta},\E}(\widehat{\bu}_I,\bv_h) - \widehat{a}^{\boldsymbol{\beta},\E}(\widehat{\bu},\bv_h)\right. \\
\displaystyle \left. + b^\E(\bv_h,\mathcal{P}_h(\widehat{p}) - \widehat{p}) + b^\E(\widehat{\bu}_I-\widehat{\bu},q_h)\right] - c(\boldsymbol{f},\bv_h - \boldsymbol{\Pi}_{0}\bv_h).
\end{multline*}
Next, using triangle inequality and invoking Lemmas \ref{lmm:bh} and \ref{estimate2} we obtain 
\begin{multline}\label{eq1}
a_h^{\nabla,\E}(\widehat{\bu}_I - \widehat{\bu}_\pi,\bv_h) \leq \nu|\widehat{\bu}_I - \widehat{\bu}_\pi|_{1,\E}|\bv_h|_{1,\E} \\
\lesssim \nu\|\bv_h\|_{1,\E}\left(|\widehat{\bu} - \widehat{\bu}_I|_{1,\E} + |\widehat{\bu} - \widehat{\bu}_\pi|_{1,\E}\right) \lesssim \nu h_\E^{s}\|\widehat{\bu}\|_{1+s,\E}\|\bv_h\|_{1,\E}.
\end{multline}
Invoking again Lemma \ref{lmm:bh} we obtain
\begin{equation}\label{eq2}
a^{\nabla,\E}(\widehat{\bu}_\pi - \widehat{\bu},\bv_h) \leq \nu|\widehat{\bu}-\widehat{\bu}_\pi|_{1,\E}|\bv_h|_{1,\E} \lesssim \nu h_\E^{s}\|\widehat{\bu}\|_{1+s,\E}\|\bv_h\|_{1,\E}.
\end{equation}
Now, invoking Lemma \ref{estimate3} and \ref{estimate2} we obtain 
\begin{equation}\label{eq3}
\begin{split}
b^\E(\bv_h,\mathcal{P}_h(\widehat{p}) - \widehat{p}) &\leq |\bv_h|_{1,\E}\|\widehat{p}-\mathcal{P}_h(\widehat{p})\|_{0,\E} \leq h_\E^{s}\|\bv_h\|_{1,\E}\|\widehat{p}\|_{s,\E},\\
b^\E(\widehat{\bu}_I-\widehat{\bu},q_h) &\leq |\widehat{\bu}-\widehat{\bu}_I|_{1,\E}\|q_h\|_{0,\E} \leq h_\E^s\|\widehat{\bu}\|_{1+s,\E}\|q_h\|_{0,\E}.
\end{split}
\end{equation}
On the other hand, using approximation properties of $\boldsymbol{\Pi}_{0}^\E$, we have
\begin{equation}\label{eq4}
c(\boldsymbol{f},\bv_h-\boldsymbol{\Pi}_{0}\bv_h) \leq \|\bb{f}\|_{0,\O}\|\bv_h - \boldsymbol{\Pi}_{0}\bv_h\|_{0,\O} \leq h\|\bb{f}\|_{1,\O}\|\bv_h\|_{1,\O}.
\end{equation}
Finally, we need to control the convective terms. To do this task, we have 
\begin{multline*}
\widehat{a}_h^{\boldsymbol{\beta},\E}(\widehat{\bu}_I,\bv_h) - \widehat{a}^{\boldsymbol{\beta},\E}(\widehat{\bu},\bv_h) = \underbrace{\widehat{a}_h^{\boldsymbol{\beta},\E}(\widehat{\bu}_I,\bv_h) - \widehat{a}^{\boldsymbol{\beta},\E}(\widehat{\bu}_I,\bv_h)}_{\textrm{(I)}} + \underbrace{\widehat{a}^{\boldsymbol{\beta},\E}(\widehat{\bu}_I-\widehat{\bu},\bv_h)}_{\textrm{(II)}}.
\end{multline*}
Observe that $\textrm{(II)}$ is easily bounded by using Lemma \ref{estimate2}, obtaining
\begin{multline*}
\textrm{(II)} \leq \dfrac{\|\bb{\beta}\|_{\infty,\E}}{2}|\widehat{\bu}-\widehat{\bu}_I|_{1,\E}\|\bv_h\|_{0,\E} + \dfrac{\|\bb{\beta}\|_{\infty,\E}}{2}|\bv_h|_{1,\E}\|\widehat{\bu}-\widehat{\bu}_I\|_{0,\E} \\ 
\lesssim \dfrac{\|\bb{\beta}\|_{\infty,\E}}{2} h_\E^s\|\widehat{\bu}\|_{1+s,\E}\|\bv_h\|_{1,\E} + \dfrac{\|\bb{\beta}\|_{\infty,\E}}{2} h_\E^{1+s}\|\widehat{\bu}\|_{1+s,\E}\|\bv_h\|_{1,\E} \\
\lesssim \|\bb{\beta}\|_{\infty,\E}h_\E^{\sigma}\|\widehat{\bu}\|_{1+s,\E}\|\bv_h\|_{1,\E}
\end{multline*}

\noindent To estimate $\textrm{(I)}$, we observe that the following identity can be obtained
\begin{equation*}
\left((\bb{\beta}\cdot\nabla)\bb{u},\bb{v}\right)_{0,\O} = \left((\nabla\bb{u})\bb{\beta},\bb{v}\right)_{0,\O} = \left(\nabla \bb{u},(\bb{\beta}\otimes\bb{v})^{t}\right)_{0,\O},
\end{equation*}
where $t$ denotes the traspose operator. On the other hand, observe that by the definition of the projectors $\Pi_{0}^{0,\E}$ and $\bb{\Pi}^{\nabla,\E}$, for every $c \in \mathbb{P}_{0}(\E)^{2\times2}$ we have
\begin{equation*}
(\Pi_0^{0,\E}\nabla \bb{v},c)_{0,\E} = (\nabla \bb{v}, c)_{0,\E} = (\nabla \bb{\Pi}^{\nabla,\E}\bb{v},c)_{0,\E},
\end{equation*}
obtaining that $\Pi_0^{0,\E}\nabla \bb{v} = \nabla \bb{\Pi}^{\nabla,\E}\bb{v}$. With the previous identities at hand, we have
\begin{multline*}
a_h^{\boldsymbol{\beta},\E}(\widehat{\bu}_I,\bv_h)-a^{\boldsymbol{\beta},\E}(\widehat{\bu}_I,\bv_h)=
\left((\nabla\boldsymbol{\Pi}^{\nabla,\E}\widehat{\bu}_I)\bb{\beta},\bb{\Pi}_0^\E\bv_h\right)_{0,\E} - \left((\nabla \widehat{\bu}_I)\bb{\beta},\bv_h)\right)_{0,\E}\\
= \left((\nabla\boldsymbol{\Pi}^{\nabla,\E} \widehat{\bu}_I - \nabla \widehat{\bu}_I)\bb{\beta}, \boldsymbol{\Pi}_0^{\E}\bv_h\right)_{0,\E} + \left((\nabla \widehat{\bu}_I)\bb{\beta}, \boldsymbol{\Pi}_0^{\E}\bv_h - \bv_h\right)_{0,\E}\\
= \left(\nabla\boldsymbol{\Pi}^{\nabla,\E} \widehat{\bu}_I - \nabla \widehat{\bu}_I, (\bb{\beta}\otimes\boldsymbol{\Pi}_0^{\E}\bv_h)^{t}\right)_{0,\E} + \left((\nabla \widehat{\bu}_I)\bb{\beta}, \boldsymbol{\Pi}_0^{\E}\bv_h - \bv_h\right)_{0,\E}\\
= \left(\nabla\boldsymbol{\Pi}^{\nabla,\E} \widehat{\bu}_I - \nabla \widehat{\bu}_I, (\bb{\beta}\otimes(\boldsymbol{\Pi}_0^{\E}\bv_h - \bv_h))^{t}\right)_{0,\E} \\
+ \left((\nabla \widehat{\bu}_I)\bb{\beta} - \boldsymbol{\Pi}_0^{\E}((\nabla \widehat{\bu}_I)\bb{\beta}), \boldsymbol{\Pi}_0^{\E}\bv_h - \bv_h\right)_{0,\E} \\ 
+ \left(\nabla\boldsymbol{\Pi}^{\nabla,\E} \widehat{\bu}_I- \nabla \widehat{\bu}_I, (\bb{\beta}\otimes\bv_h)^{t} - \Pi_0^{0,\E}(\bb{\beta}\otimes\bv_h)^{t}\right)_{0,\E}.
\end{multline*}
Hence, using the Cauchy-Schwarz inequality we obtain
\begin{multline*}
a_h^{\boldsymbol{\beta},\E}(\widehat{\bu}_I,\bv_h) - a^{\boldsymbol{\beta},\E}(\widehat{\bu}_I,\bv_h) \leq \|(\nabla \widehat{\bu}_I)\boldsymbol{\beta} - \boldsymbol{\Pi}_0^{\E}((\nabla \widehat{\bu}_I)\boldsymbol{\beta})\|_{0,\E}\|\bv_h - \boldsymbol{\Pi}_0^{\E}\bv_h\|_{0,\E} \\ 
+ |\widehat{\bu}_I - \boldsymbol{\Pi}^{\nabla,\E} \widehat{\bu}_I|_{1,\E}\left(\|\boldsymbol{\beta}\otimes(\boldsymbol{\Pi}_0^{\E}\bv_h - \bv_h)\|_{0,\E} + \|\boldsymbol{\beta}\otimes\bv_h - \boldsymbol{\Pi}_0^{\E}(\boldsymbol{\beta}\otimes\bv_h)^{t}\|_{0,\E}\right) \\
\lesssim h_\E\|\boldsymbol{\beta}\|_{\infty,\E}|\widehat{\bu}_I|_{1,\E}\|\bv_h\|_{1,\E} \lesssim h_\E\|\boldsymbol{\beta}\|_{\infty,\E}\|\widehat{\bu}\|_{1,\E}\|\bv_h\|_{1,\E},
\end{multline*}
where in the last inequality we have used the stability of $\boldsymbol{\Pi}^{\nabla,\E}$, Cauchy-Schwarz inequality and error estimates for $\boldsymbol{\Pi}_0^\E$, together with the assumption that $(\bb{\beta}\otimes\bv_h)^{t}$ has suffient regularity. Changing the roles of $\widehat{\bb{u}}_I$ and $\bb{v}_h$ we conclude that 
\begin{equation}\label{eq5}
\widehat{a}_h^{\boldsymbol{\beta},\E}(\widehat{\bu}_I,\bv_h) - \widehat{a}^{\boldsymbol{\beta},\E}(\widehat{\bu}_I,\bv_h) \lesssim h_\E^\sigma\|\boldsymbol{\beta}\|_{\infty,\E}\|\widehat{\bu}\|_{1,\E}\|\bv_h\|_{1,\E},
\end{equation}
where $\sigma := \min\{1,s\}$. Therefore, summing over all polygons, using the estimates \eqref{eq:estimatefuente_velocity}, \eqref{eq:estimatefuente_pressure}, invoking Theorem \ref{th:regularidadfuente} and gathering \eqref{eq1}, \eqref{eq2}, \eqref{eq3}, \eqref{eq4}, \eqref{eq5} we obtain 
\begin{equation*}
\|(\bT-\bT_h)\boldsymbol{f}\|_{1,\O} \leq \bb{C}h^{\sigma}\|\boldsymbol{f}\|_{1,\O},
\end{equation*}
where if $C>0$ is the constant defined in Theorem \ref{th:regularidadfuente}, then 
\begin{equation}
\label{eq:constant1}
\bb{C} := \max\left\lbrace 2C, 2C\nu, C\|\bb{\beta}\|_{\infty,\O}, \dfrac{2C_{pf}}{\nu}\|\bb{\beta}\|_{\infty,\O}, 1\right\rbrace.
\end{equation}
This concludes the proof.
\end{proof}
Also for the adjoint problem, we have the same convergence result. Since the proof is essentially identical to Lemma \ref{lmm:conv1} we skip the steps of the proof.
\begin{lemma}
\label{eq:adjoint_diff}
There exists a constant $\bb{C}>0$, independent of $h$, such that
\begin{equation*}
\|(\bT^*-\bT_h^*)\boldsymbol{f}\|_{1,\O}\leq \bb{C}h^{\sigma^*}\|\boldsymbol{f}\|_{1,\O}, \quad \sigma^* := \min\{1,s^{*}\},
\end{equation*}
where $\bb{C}>0$ is defined in Lemma \ref{lmm:conv1}.
\end{lemma}

Now we are in position to apply the theory of  \cite{MR0203473}  to conclude that  our numerical method does not introduce spurious eigenvalues. This is stated in the following theorem.
\begin{theorem}
	\label{thm:spurious_free}
	Let $V\subset\mathbb{C}$ be an open set containing $\sp(\bT)$. Then, there exists $h_0>0$ such that $\sp(\bT_h)\subset V$ for all $h<h_0$.
\end{theorem}
\subsection{Error estimates}
\label{sec:conv}
To do the task of deriving error estimates, we begin by  recalling some definitions. Let $\kappa$ be a nonzero isolated eigenvalue of $\bT$ with algebraic multiplicity $m$ and let $\Gamma$
be a disk of the complex plane centered in $\kappa$, such that $\kappa$ is the only eigenvalue of $\bT$ lying in $\Gamma$ and $\partial\Gamma\cap\sp(\bT)=\emptyset$. We define the spectral projections of $\boldsymbol{E}$ and $\boldsymbol{E}^*$, associated to $\bT$ and $\bT^*$, respectively, as follows:
\begin{enumerate}
\item The spectral projector of $\bT$ associated to $\kappa$ is $\displaystyle \boldsymbol{E}:=\frac{1}{2\pi i}\int_{\partial\Gamma} (z\boldsymbol{I}-\bT)^{-1}\,dz;$
\item The spectral projector of $\bT^*$ associated to $\overline{\kappa}$ is $\displaystyle \boldsymbol{E}^*:=\frac{1}{2\pi i}\int_{\partial\Gamma} (z\boldsymbol{I}-\bT^*)^{-1}\,dz,$
\end{enumerate}
where $\boldsymbol{I}$ represents the identity operator. Let us remark that $\boldsymbol{E}$ and $\boldsymbol{E}^*$ are the projections onto the generalized eigenvector $R(\boldsymbol{E})$ and $R(\boldsymbol{E}^*)$, respectively. 

A consequence of Lemma \ref{lmm:conv1} is that there exist $m$ eigenvalues that  lie in $\Gamma$, namely $\kappa_h^{(1)},\ldots,\kappa_h^{(m)}$, repeated according their respective multiplicities, that converge to $\kappa$ as $h$ goes to zero. With this result at hand, we introduce the following spectral projection
\begin{equation*}
\boldsymbol{E}_h:=\frac{1}{2\pi i}\int_{\partial\Gamma} (z\boldsymbol{I}-\bT_h)^{-1}\,dz,
\end{equation*}
which is a projection onto the discrete invariant subspace $R(\boldsymbol{E}_h)$ of $\bT$, spanned by the generalized eigenvector of $\bT_h$ corresponding to 
 $\kappa_h^{(1)},\ldots,\kappa_h^{(m)}$.
Now we recall the definition of the \textit{gap} $\hdel$ between two closed
subspaces $\CM$ and $\CN$ of $\L^2(\O)$:
$$
\hdel(\CM,\CN)
:=\max\big\{\delta(\CM,\CN),\delta(\CN,\CM)\big\}, \text{ where } \delta(\CM,\CN)
:=\sup_{\underset{\left\|x\right\|_{0,\O}=1}{x\in\CM}}
\left(\inf_{y\in\CN}\left\|x-y\right\|_{0,\O}\right).
$$
We end this section proving error estimates for the eigenfunctions and eigenvalues.
\begin{theorem}
\label{thm:errors1}
The following estimates hold
\begin{equation*}
\hdel(R(\boldsymbol{E}),R(\boldsymbol{E}_h))\leq \bb{C}h^{\sigma},\quad\hdel(R(\boldsymbol{E}^*),R(\boldsymbol{E}_h^*))\leq \bb{C}h^{\sigma^*}\quad\text{and}\quad
|\kappa-\widehat{\kappa}_h|\leq \widetilde{\bb{C}}h^{\sigma+\sigma^*},
\end{equation*}
where $\displaystyle \widehat{\kappa}_{h} := \frac{1}{m}\sum_{i=1}^{m} \kappa_{h}^{(i)}$, $\widetilde{\bb{C}}>0$ is independent of $h$ and $\bb{C}>0$ is defined in Lemma \ref{lmm:conv1}.
\end{theorem}
\begin{proof}
The proof of the gap between the eigenspaces is a direct consequence of the convergence in norm between $\bT$ and $\bT_h$ as $h$ goes to zero.
We focus on the double order of convergence for the eigenvalues. Let $\{\bu_k\}_{k=1}^m$ be such that $\bT \bu_k=\kappa \bu_k$, for $k=1,\ldots,m$, and let $\{\bu_k^*\}_{k=1}^m$ be a dual basis for $\{\bu_k\}_{k=1}^m$. This basis satisfies $A((\bu_k,p),(\bu_l^*,p^*))=\delta_{k.l},$
where $\delta_{k.l}$ represents the Kronecker delta.
From \cite[Theorem 7.2]{BO}, the following identity holds true
\begin{equation*}
|\kappa - \widehat{\kappa}_{h}| \lesssim \displaystyle \frac{1}{m}\sum_{i=1}^{m} \left\lvert \left\langle (\bT-\bT_h)\bu_k, \bu_k^{*}\right\rangle \right\vert + \|\bT - \bT_h\|_{\mathcal{L}(\H_0^1(\O,\mathbb{C}))}\|\bT^* - \bT_h^*\|_{\mathcal{L}(\H_0^{1}(\O,\mathbb{C}))},
\end{equation*}
where $\langle\cdot,\cdot\rangle$ denotes the corresponding duality pairing. Observe that the last term is estimate invoking Lemma \ref{lmm:conv1}, obtaining
\begin{equation*}
\|\bT - \bT_h\|_{\mathcal{L}(\H_0^1(\O,\mathbb{C}))}\|\bT^* - \bT_h^*\|_{\mathcal{L}(\H_0^{1}(\O,\mathbb{C}))} \leq \bb{C}^{2}h^{\sigma+\sigma^*}.
\end{equation*}
Hence, we only need to control the first term in the above estimate. In order to obtain such a estimate, observe that the following identity holds
\begin{multline*}
\left\langle (\bT-\bT_h)\bu_k, \bu_k^{*}\right\rangle = A(((\bT-\bT_h)\bu_k,p-p_h),(\bu_k^*,p^*)) \\
= A(((\bT-\bT_h)\bu_k,p-p_h),(\bu_k^* - (\bu_k^*)_I,p^* - \mathcal{P}_h(p^*))) \\
+ A(((\bT-\bT_h)\bu_k,p-p_h),((\bu_k^*)_I,\mathcal{P}_h(p^*))) \\
= A(((\bT-\bT_h)\bu_k,p-p_h),(\bu_k^* - (\bu_k^*)_I,p^* - \mathcal{P}_h(p^*))) \\
+ A((\bT\bu_k,p),((\bu_k^*)_I,\mathcal{P}_h(p^*))) - A((\bT_h\bu_k,p_h),((\bu_k^*)_I,\mathcal{P}_h(p^*))) \\
= \underbrace{A(((\bT-\bT_h)\bu_k,p-p_h),(\bu_k^* - (\bu_k^*)_I,p^* - \mathcal{P}_h(p^*)))}_{\textrm{(I)}} \\
+ \underbrace{A_h((\bT_h\bu_k,p_h),((\bu_k^*)_I,\mathcal{P}_h(p^*))) - A((\bT_h\bu_k,p_h),((\bu_k^*)_I,\mathcal{P}_h(p^*)))}_{\textrm{(II)}} \\
+ \underbrace{\left[c(\bu_k,(\bu_k^*)_I) - c_h(\bu_k,(\bu_k^*)_I)\right]}_{\textrm{(III)}}.
\end{multline*}
Now, our task is to estimate the contributions $\textrm{(I)}$, $\textrm{(II)}$ and $\textrm{(III)}$. For $\textrm{(I)}$, using the Cauchy-Schwarz inequality and invoking Lemmas \ref{lmm:conv1}, \ref{estimate2} and \ref{estimate3} we obtain
\begin{multline}\label{eqI}
\textrm{(I)} \leq \nu\|(\bT-\bT_h)\bu_k\|_{1,\O}|\bu_k^* - (\bu_k^*)_I|_{1,\O} + \|\boldsymbol{\beta}\|_{\infty,\O}\|(\bT-\bT_h)\bu_k\|_{1,\O}\|\bu_k^*-(\bu_k^*)_I\|_{0,\O} \\ 
+ \|p-p_h\|_{0,\O}|\bu_k^* - (\bu_k^*)_I|_{1,\O} + \|p^* - \mathcal{P}_h(p^*)\|_{0,\O}\|(\bT-\bT_h)\bu_k\|_{1,\O} \\ 
\leq \max\{\bb{C}\|\boldsymbol{\beta}\|_{\infty,\O},\bb{C}\nu,\bb{C}\}h^{\sigma+\sigma^*}. 
\end{multline}
For $\textrm{(III)}$, using the definition of $\boldsymbol{\Pi}_{0}$ and it error estimates, together with the stability of the virtual interpolant, we obtain
\begin{multline}\label{eqIII}
\textrm{(III)} = c(\bu_k-\boldsymbol{\Pi}_0,(\bu_k^*)_I-\boldsymbol{\Pi}_0(\bu_k^*)_I) \\
\leq \|\bu_k-\boldsymbol{\Pi}_0\bu_k\|_{0,\O}\|(\bu_k^*)_I-\boldsymbol{\Pi}_0(\bu_k^*)_I)\|_{0,\O} \leq h^{\sigma+\sigma^*}.
\end{multline}
Finally, we need to estimate $\textrm{(II)}$. For this, first we note that
\begin{multline*}
\textrm{(II)} = \displaystyle \sum_{\E\in\CT_h} \left[a_h^{\nabla,\E}(\bT_h\bu_k,(\bu_k^*)_I) - a^{\nabla,\E}(\bT_h\bu_k,(\bu_k^*)_I) \right] \\ 
+ \left[\widehat{a}_h^{\boldsymbol{\beta},\E}(\bT_h\bu_k,(\bu_k^*)_I) - \widehat{a}^{\boldsymbol{\beta},\E}(\bT_h\bu_k,(\bu_k^*)_I) \right].
\end{multline*}
Hence, we need to control the terms associated to the gradients and the convective terms. So, we have
\begin{multline*}
\displaystyle \sum_{\E\in\CT_h} \left[a_h^{\nabla,\E}(\bT_h\bu_k,(\bu_k^*)_I) - a^{\nabla,\E}(\bT_h\bu_k,(\bu_k^*)_I) \right] \\
\leq C\nu \displaystyle \sum_{\E\in\CT_h} |\bT_h\bu_k - \boldsymbol{\Pi}^{\nabla,\E}\bT_h\bu_k|_{1,\E}|(\bu_k^*)_I - \boldsymbol{\Pi}^{\nabla,\E}(\bu_k^*)_I|_{1,\E} \\
\leq C\nu|\bT_h\bu_k - \boldsymbol{\Pi}^{\nabla}\bT_h\bu_k|_{1,h}|(\bu_k^*)_I - \boldsymbol{\Pi}^{\nabla}(\bu_k^*)_I|_{1,h},
\end{multline*}
and adding and subtracting $\bT\bu_k$ and $\boldsymbol{\Pi}^{\nabla}\bT\bu_k$ in $|\bT_h\bu_k - \boldsymbol{\Pi}^{\nabla}\bT_h\bu_k|_{1,h}$, and $\bu_k$ and $\boldsymbol{\Pi}^{\nabla}\bu_k^*$ in $|(\bu_k^*)_I - \boldsymbol{\Pi}^{\nabla}(\bu_k^*)_I|_{1,h}$, together with triangle inequality, we obtain
\begin{equation*}
\displaystyle \sum_{\E\in\CT_h} \left[a_h^{\nabla,\E}(\bT_h\bu_k,(\bu_k^*)_I) - a^{\nabla,\E}(\bT_h\bu_k,(\bu_k^*)_I) \right] \leq \max\{\bb{C}\nu,\bb{C}\} h^{\sigma+\sigma^*}, 
\end{equation*}
where in the last inequality we have invoked Lemmas \ref{lmm:conv1}, \ref{estimate2} and \ref{lmm:bh}. For the convective terms, we proceed as in the proof of Lemma \ref{eq:infsupdiscreta}, obtaining
\begin{multline*}
\displaystyle \sum_{\E\in\CT_h} \left[a_h^{\boldsymbol{\beta},\E}(\bT_h\bu_k,(\bu_k^*)_I) - a^{\boldsymbol{\beta},\E}(\bT_h\bu_k,(\bu_k^*)_I) \right]\\
 \leq 2\|(\nabla \bT_h\bu_k)\bb{\beta} - \boldsymbol{\Pi}_0((\nabla \bT_h\bu_k)\bb{\beta})\|_{0,\O}\|(\bu_k^*)_I - \boldsymbol{\Pi}_0(\bu_k^*)_I\|_{0,\O} \\ 
+ 2|\bT_h\bu_k - \boldsymbol{\Pi}^\nabla \bT_h\bu_k|_{1,h}\left(\|(\boldsymbol{\beta}\otimes(\boldsymbol{\Pi}_0(\bu_k^*)_I - (\bu_k^*)_I))^{t}\|_{0,\O}\right. \\ 
\left. + \|(\boldsymbol{\beta}\otimes(\bu_k^*)_I)^{t} - \boldsymbol{\Pi}_0(\boldsymbol{\beta}\otimes(\bu_k^*)_I)^{t}\|_{0,\O}\right).
\end{multline*}
Let us focus  on the term $\|(\nabla \bT_h\bu_k)\bb{\beta} - \boldsymbol{\Pi}_0((\nabla \bT_h\bu_k)\bb{\beta})\|_{0,\O}$ of the previous  estimate. On this term we add and subtract the terms  $(\nabla \bT\bu_k)\bb{\beta}$ and $\boldsymbol{\Pi}_0((\nabla \bT\bu_k)\bb{\beta})$. Now, for the term $|\bT_h\bu_k - \boldsymbol{\Pi}^\nabla \bT_h\bu_k|_{1,\O}$ we  add and subtract $\bT\bu_k$ and $\boldsymbol{\Pi}^\nabla \bT\bu_k$. Then, applying triangle inequality in each term,  together with Lemmas \ref{lmm:conv1} and the error estimates for $\boldsymbol{\Pi}_0$, yields to
\begin{equation*}
\displaystyle \sum_{\E\in\CT_h} \left[a_h^{\boldsymbol{\beta},\E}(\bT_h\bu_k,(\bu_k^*)_I) - a^{\boldsymbol{\beta},\E}(\bT_h\bu_k,(\bu_k^*)_I) \right] \leq 2\max\{\bb{C}\|\bb{\beta}\|_{\infty,\O},\|\bb{\beta}\|_{\infty,\O}\}h^{\sigma+\sigma^*}.
\end{equation*} 
Then, changing the roles of $\bb{T}_h\bb{u}_k$ and $(\bb{u}_k^*)_I$ and proceeding in a similar way, we conclude that 
\begin{equation*}
\displaystyle \sum_{\E\in\CT_h} \left[\widehat{a}_h^{\boldsymbol{\beta},\E}(\bT_h\bu_k,(\bu_k^*)_I) - \widehat{a}^{\boldsymbol{\beta},\E}(\bT_h\bu_k,(\bu_k^*)_I) \right] \leq 2\max\{\bb{C}\|\bb{\beta}\|_{\infty,\O},\|\bb{\beta}\|_{\infty,\O}\}h^{\sigma+\sigma^*}.
\end{equation*}
Therefore, we obtain that
\begin{equation}\label{eqII}
\textrm{(II)} \leq \max\{\bb{C}\|\boldsymbol{\beta}\|_{\infty,\O}, \bb{C}\nu, \|\bb{\beta}\|_{\infty,\O}, \bb{C}\}h^{\sigma+\sigma^*}.
\end{equation}

Finally, gathering \eqref{eqI}, \eqref{eqIII} and \eqref{eqII}, and defining 
\begin{equation*}
\widetilde{\bb{C}} := \max\{\bb{C}\|\boldsymbol{\beta}\|_{\infty,\O},\bb{C}\nu,\|\bb{\beta}\|_{\infty,\O},\bb{C},1\},
\end{equation*}
where $\bb{C}$ is defined in \eqref{eq:constant1}, we conclude the proof.
\end{proof}

Our next goal is to obtain an $\L^{2}$ norm estimate for the velocity. To do this task we use a standard duality argument.
\begin{lemma}\label{lmm:duality1}
Let $\boldsymbol{f} \in R(\boldsymbol{E})$ be such that $\widehat{\bu} := \bT\boldsymbol{f}$ and $\widehat{\bu}_h := \bT_h\boldsymbol{f}$. Then, the following estimate holds
\begin{equation*}
\|\widehat{\bu}-\widehat{\bu}_h\|_{0,\O} \leq \bb{\mathtt{C}}h^{\widetilde{r}+\sigma}\|\boldsymbol{f}\|_{1,\O},
\end{equation*}
where $\bb{\mathtt{C}}>0$ is independent of $h$.
\end{lemma}

\begin{proof}
Let us consider the following auxiliary problem: Find $(\bb{z},\phi) \in \mathcal{X}$ such that 
\begin{equation}\label{eq:auxiliar}
	\left\{
	\begin{array}{rcll}
a(\bv,\bb{z}) + b(\bv,\phi)&=& c(\bv,\widehat{\bu}-\widehat{\bu}_h)&\forall \bv\in \H_0^1(\O,\mathbb{C})^2,\\
b(\bb{z},q)&=&0&\forall q\in \L_0^2(\O,\mathbb{C}).
\end{array}
	\right.
\end{equation}
Observe that this problem is well-posed and its solution satisfies the following estimate
\begin{equation}\label{eq:additionalreg_aux}
\|\bb{z}\|_{1+r,\O} + \|\phi\|_{r,\O} \leq C\|\widehat{\bu}-\widehat{\bu}_h\|_{0,\O},
\end{equation}
where $C>0$ is the constant given by Theorem \ref{th:regularidadfuente}. Moreover, $\bb{z}$ and $\phi$ satisfies the estimates \eqref{eq:estimatefuente_velocity} and \eqref{eq:estimatefuente_pressure} respectively. Now we need to estimate the term on the right-hand side of \eqref{eq:additionalreg_aux}. First, we define $\widetilde{r} := \min\{1,r\}$. Now, testing \eqref{eq:auxiliar} with $\bv := \widehat{\bu}-\widehat{\bu}_h$ we obtain
\begin{equation*}
\begin{split}
\|\widehat{\bu}-\widehat{\bu}_h\|_{0,\O}^{2} &= c(\widehat{\bu}-\widehat{\bu}_h,\widehat{\bu}-\widehat{\bu}_h) \\
&= a(\widehat{\bu}-\widehat{\bu}_h,\bb{z}) + b(\widehat{\bu}-\widehat{\bu}_h,\phi) \\
&= \underbrace{a(\widehat{\bu}-\widehat{\bu}_h,\bb{z}-\bb{z}_I)}_{\textrm{(I)}} + \underbrace{b(\widehat{\bu}-\widehat{\bu}_h,\phi-\mathcal{P}_h(\phi))}_{\textrm{(II)}} + \underbrace{a(\widehat{\bu}-\widehat{\bu}_h,\bb{z}_I)}_{\textrm{(III)}},
\end{split}
\end{equation*}
where in the last equality we have used that $b(\widehat{\bu}-\widehat{\bu}_h,\mathcal{P}_h(\phi)) = 0$ from the second equations of \eqref{def:oseen_system_weak_source} and \eqref{eq:weak_oseen_source_disc}. 

The task now is to estimate each of the contributions $\textrm{(I)}$, $\textrm{(II)}$ and $\textrm{(III)}$ on the above identity. For $\textrm{(I)}$, using the  Cauchy-Schwarz inequality and invoking Lemmas \ref{lmm:conv1} and \ref{estimate2} we have
\begin{multline}\label{eq:estimateI}
\textrm{(I)} \leq \nu\|\widehat{\bu}-\widehat{\bu}_h\|_{1,\O}|\bb{z}-\bb{z}_I|_{1,\O} \leq \bb{C}\nu h^{\widetilde{r}+\sigma}\|\bb{f}\|_{1,\O}\|\bb{z}\|_{1+\widetilde{r},\O} \\
\leq \bb{C}C\nu h^{\widetilde{r}+\sigma}\|\bb{f}\|_{1,\O}\|\widehat{\bu}-\widehat{\bu}_h\|_{0,\O}.
\end{multline}
For $\textrm{(II)}$, we invoke Lemmas \ref{lmm:conv1} and \ref{estimate3}, obtaining
\begin{multline}\label{eq:estimateII}
\textrm{(II)} \leq \|\widehat{\bu}-\widehat{\bu}_h\|_{1,\O}\|\phi-\mathcal{P}_h(\phi)\|_{0,\O} \leq \bb{C}h^{\widetilde{r}+\sigma}\|\bb{f}\|_{1,\O}\|\phi\|_{\widetilde{r},\O} \\ 
\leq \bb{C}Ch^{\widetilde{r}+\sigma}\|\bb{f}\|_{1,\O}\|\widehat{\bu}-\widehat{\bu}_h\|_{0,\O}.
\end{multline}
Finally, observe that for $\textrm{(III)}$, using the first equations in \eqref{def:oseen_system_weak_source} and \eqref{eq:weak_oseen_source_disc}, the following identity can be obtained
\begin{equation*}
\begin{split}
a(\widehat{\bu}-\widehat{\bu}_h,\bb{z}_I) &= a(\widehat{\bu},\bb{z}_I) - a(\widehat{\bu}_h,\bb{z}_I) \\
&= c(\bb{f},\bb{z}_I) - b(\bb{z}_I,p) - a(\widehat{\bu}_h,\bb{z}_I) \\
&= \underbrace{c(\bb{f},\bb{z}_I) - c_h(\bb{f},\bb{z}_I)}_{\textrm{(IV)}} + \underbrace{b(\bb{z}-\bb{z}_I,p-p_h)}_{\textrm{(V)}} + \underbrace{a_h(\widehat{\bu}_h,\bb{z}_I) - a(\widehat{\bu}_h,\bb{z}_I)}_{\textrm{(VI)}}.
\end{split}
\end{equation*}
Note that the contributions $\textrm{(IV)}$ and $\textrm{(V)}$ can be easily estimated, obtaining 
\begin{equation*}
\begin{split}
\textrm{(IV)} &\leq \|\bb{f} - \bb{\Pi}_0\bb{f}\|_{0,\O}\|\bb{z}_I-\bb{\Pi}_0\bb{z}_I\|_{0,\O} \leq h^{\widetilde{r}+\sigma}\|\bb{f}\|_{1,\O}\|\bb{z}\|_{1,\O}, \\
\textrm{(V)} &\leq \|p-p_h\|_{0,\O}|\bb{z}-\bb{z}_I|_{1,\O} \leq \bb{C}h^{\widetilde{r}+\sigma}\|\bb{f}\|_{1,\O}\|\bb{z}\|_{1+\widetilde{r},\O},
\end{split}
\end{equation*}
whereas for $\textrm{(VI)}$ we have
\begin{equation*}
\textrm{(VI)} \leq 2\max\{\bb{C}\|\bb{\beta}\|_{\infty,\O},\bb{C}\nu,\|\bb{\beta}\|_{\infty,\O},1\}h^{\widetilde{r}+\sigma}\|\bb{f}\|_{1,\O}\|\bb{z}\|_{1,\O}.
\end{equation*}
Hence, we obtain for $\textrm{(III)}$ that
\begin{equation}\label{eq:estimateIII}
\textrm{(III)} \leq \widetilde{C}^*h^{\widetilde{r}+\sigma}\|\bb{f}\|_{1,\O}\|\bb{z}\|_{1+\widetilde{r},\O},
\end{equation}
where 
\begin{equation*}
\widehat{C}^* := \max\left\lbrace \dfrac{C_{pf}}{\nu}, \bb{C}C, 2\max\{\bb{C}\|\bb{\beta}\|_{\infty,\O}, \bb{C}\nu,\|\bb{\beta}\|_{\infty,\O},1\}\dfrac{C_{pf}}{\nu}\right\rbrace,
\end{equation*}
and gathering \eqref{eq:estimateI}, \eqref{eq:estimateII} and \eqref{eq:estimateIII}, together with \eqref{eq:additionalreg_aux}, we conclude that 
\begin{equation*}
\|\widehat{\bu}-\widehat{\bu}_h\|_{0,\O} \leq \textbf{C}h^{\widetilde{r}+\sigma}\|\bb{f}\|_{1,\O}, \quad \mathtt{C} := \max\{\bb{C}C\nu,\bb{C}C,\widetilde{C}^*\}.
\end{equation*}
This completes the proof.
\end{proof}

Observe that similar arguments can be done for the dual problem. Since the proof is identically, we skip the proof. 
\begin{lemma}
Let $\boldsymbol{f} \in R(\boldsymbol{E}^*)$ be such that $\widehat{\bu}^* := \bT^*\boldsymbol{f}$ and $\widehat{\bu}_h^* := \bT_h^*\boldsymbol{f}$. Then, the following estimate holds
\begin{equation*}
\|\widehat{\bu}^*-\widehat{\bu}_h^*\|_{0,\O} \leq \mathtt{C}h^{\widetilde{r}^*+\sigma^*}\|\boldsymbol{f}\|_{1,\O},
\end{equation*}
where $\mathtt{C}>0$  is defined in Lemma \ref{lmm:duality1}.
\end{lemma}
\subsection{Error in $\L^2$ norm}
Our aim is to derive an error estimate for the $\L^2$ norm of the velocity. To do this task, we define the solution operator on the space $\L^{2}(\O,\mathbb{C})^2$ as $\widetilde{\bT} : \L^{2}(\O,\mathbb{C})^2 \rightarrow \L^{2}(\O,\mathbb{C})^2$, which is defined by $\widetilde{\bT}\bb{f} := \widetilde{\bu}$, where $\widetilde{\bu}$ is solution of \eqref{def:oseen_system_weak_source}. Observe that $\widetilde{\bT}$ is well-defined and compact, but not self-adjoint. Hence, we define the dual operator of $\widetilde{\bT}$ as $\widetilde{\bT}^* : \L^2(\O,\mathbb{C})^2 \rightarrow \L^2(\O,\mathbb{C})^2$ as $\widetilde{\bT}^*\bb{f} := \widetilde{\bu}^*$, where $\widetilde{\bu}^*$ is solution of \eqref{def:oseen_system_weak_dual_source}. Finally, we remark that the spectra of $\widetilde{\bT}$ and $\bT$ coincide. 

Now we present the following result, in which the convergence of $\bT_h$ to $\widetilde{\bT}$ as $h\downarrow 0$ is obtained.
\begin{lemma}\label{eq:duality2}
For every $\bb{f} \in \L^2(\O,\mathbb{C})^2$, the following estimate holds
\begin{equation*}
\|(\widetilde{\bT}-\bT_h)\bb{f}\|_{0,\O} \leq \bb{C}h^{\sigma}\|\bb{f}\|_{0,\O},
\end{equation*}
where $\bb{C}>0$ is defined in Lemma \ref{lmm:conv1}.
\end{lemma}

\begin{proof}
The proof follows the same steps of those in the proof of Lemma \ref{lmm:conv1}, but considering the source $\bb{f} \in \L^{2}(\O,\mathbb{C})^2$.
\end{proof}

Similar arguments can be used in order to obtain the convergence of $\bT_h^*$ to $\widetilde{\bT}^{*}$ as $h\downarrow 0$, which we present in the following result.

\begin{lemma}
\label{eq:duality3} 
For every $\bb{f} \in \L^2(\O,\mathbb{C})^2$, the following estimate holds
\begin{equation*}
\|(\widetilde{\bT}^*-\bT_h^*)\bb{f}\|_{0,\O} \leq \bb{C}h^{\sigma^*}\|\bb{f}\|_{0,\O},
\end{equation*}
where $\bb{C}>0$ is defined in Lemma \ref{lmm:conv1}.
\end{lemma}

Finally, we conclude this section with error estimates for primal and dual eigenfunctions.

\begin{lemma}
Let $\bu_h$ be an eigenfunction of $\bT_{h}$ associated with the eigenvalue $\kappa_{h}^{(i)}$, $1\leq i \leq m$, with $\|\bu_{h}\|_{0,\O} = 1$. Then, there exists an eigenfunction $\bu \in \L^{2}(\O,\mathbb{C})^2$ of $\bT$ associated to the eigenvalue $\kappa$ such that
\begin{equation*}
\|\bu-\bu_{h}\|_{0,\O} \leq \bb{C}h^{\widetilde{r}+\sigma},
\end{equation*}
where $\bb{C}>0$ is independent of $h$.
\end{lemma}
\begin{proof}
Applying Lemma \ref{eq:duality2} and \cite[Theorem 7.1]{BO}, we have spectral convergence of $\bT_{h}$ to $\widetilde{\bT}$. Now, due to the relation between the eigenfunctions of $\bT$ and $\bT_{h}$ with those of $\widetilde{\bT}$ and $\bT_{h}$, we have $\bu_{h} \in R(\bb{E}_h)$ and there exists $\bu \in R(\bb{E})$ such that 
\begin{equation*}
\|\bu-\bu_{h}\|_{0,\O} \lesssim \underset{\widetilde{\bb{f}} \in R(\widetilde{\bb{E}}_h) : \|\widetilde{\bb{f}}\|_{0,\O}=1}{\sup} \|(\widetilde{\bT}-\bT_h)\widetilde{\bb{f}}\|_{0,\O},
\end{equation*}
where $R(\widetilde{\bb{E}}_h)$ is an eigenspace of $\widetilde{\bT}$. On the other hand, using Lemma \ref{lmm:duality1}, for all $\widetilde{\bb{f}} \in R(\widetilde{\bb{E}})$, if $\bb{f} \in R(\widetilde{\bb{E}})$ is such that $\bb{f} = \widetilde{\bb{f}}$, then 
\begin{equation*}
\|(\widetilde{\bT}-\bT_h)\widetilde{\bb{f}}\|_{0,\O} = \|(\bT-\bT_h)\bb{f}\|_{0,\O} \leq \bb{C}h^{\widetilde{r}+\sigma}.
\end{equation*}
This concludes the proof.
\end{proof}
Observe that the same arguments can be used for the dual problem, obtaining the following result.
{\begin{lemma}
Let $\bu_h^*$ be an eigenfunction of $\bT_{h}^*$ associated with the eigenvalue $\kappa_{h}^{*(i)}$, $1\leq i \leq m$, with $\|\bu_{h}^*\|_{0,\O} = 1$. Then, there exists an eigenfunction $\bu^* \in \L^{2}(\O,\mathbb{C})^2$ of $\bT^*$ associated to the eigenvalue $\kappa^*$ such that
\begin{equation*}
\|\bu^*-\bu_{h}^*\|_{0,\O} \leq \bb{C}h^{\widetilde{r}^*+\sigma^*},
\end{equation*}
where $\bb{C}>0$ is independent of $h$.
\end{lemma}

\section{Numerical experiments}
\label{sec:numerical-experiments}
We conclude our paper reporting some numerical tests to illustrate the performance of our method. The implementation of the method has been developed
  in MATLAB. The goal is to assess the performance of the method on different domains with polygonal meshes and study the presence of spurious eigenvalues. After computing the eigenvalues, the rates of convergence  are calculated by using a least-square fitting. More precisely, if $\lambda_h$ is a discrete complex eigenvalue, then the rate of convergence $\alpha$ is calculated by extrapolation with the least square  fitting
$$
\lambda_{h}\approx \lambda_{\text{extr}} + Ch^{\alpha},
$$
where $\lambda_{\text{extr}}$ is the extrapolated eigenvalue given by the fitting. Also, $\mathrm{N}$ represents the number of polygons in the bottom of the square.

For the tests we consider the following families of polygonal meshes which satisfy the assumptions  \textbf{A1} and  \textbf{A2}  (see Figure \ref{fig:mesh}):
\begin{itemize}
\item $\CT_h^{1}$: distorted squares meshes;
\item $\CT_h^2$: Voronoi meshes;
\item $\CT_h^3$: trapezoidal meshes;
\item $\CT_h^4$: polygons with middle points;
\item $\CT_h^5$: triangle mesh with middle points.
\end{itemize}
\begin{figure}[h!]
\begin{center}
\centering\includegraphics[height=4.2cm, width=4.2cm]{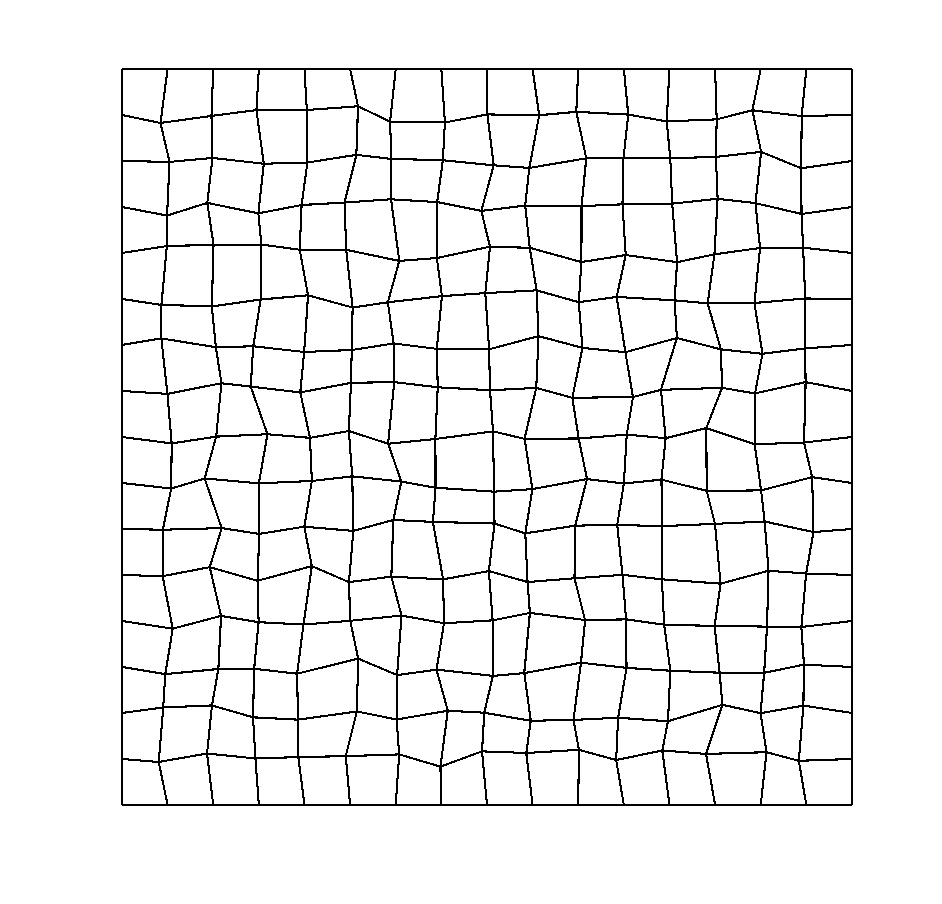} 
\centering\includegraphics[height=4.2cm, width=4.2cm]{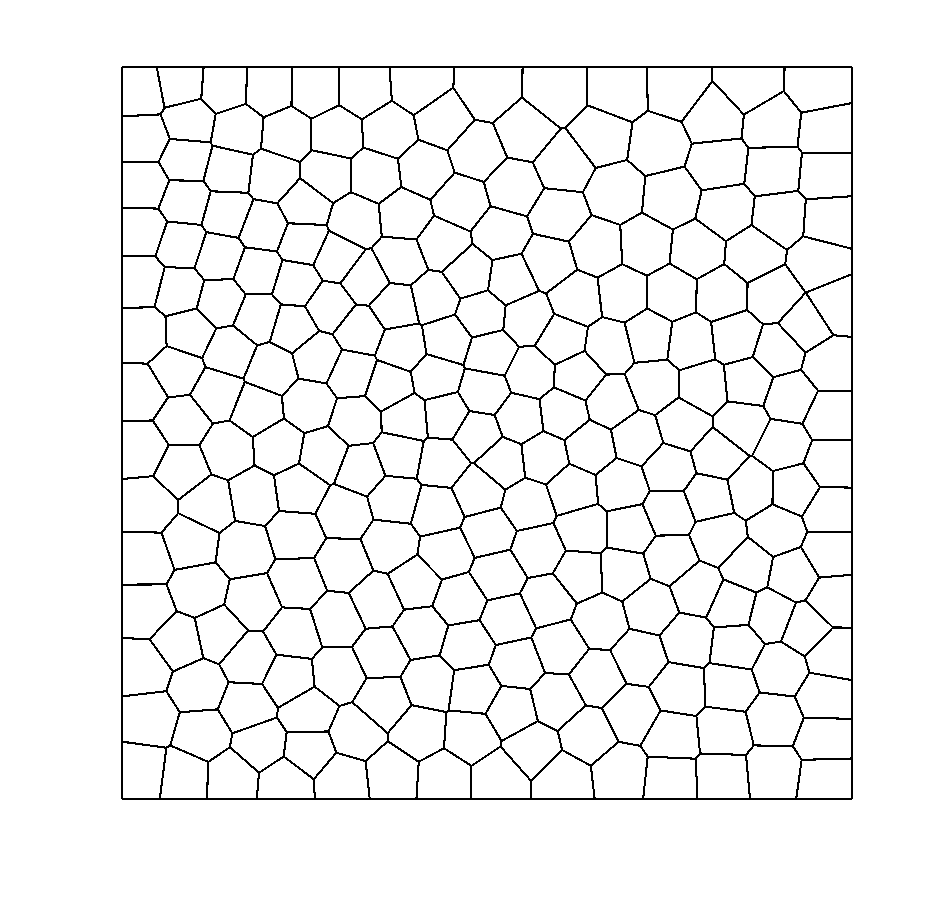}
\centering\includegraphics[height=4.2cm, width=4.2cm]{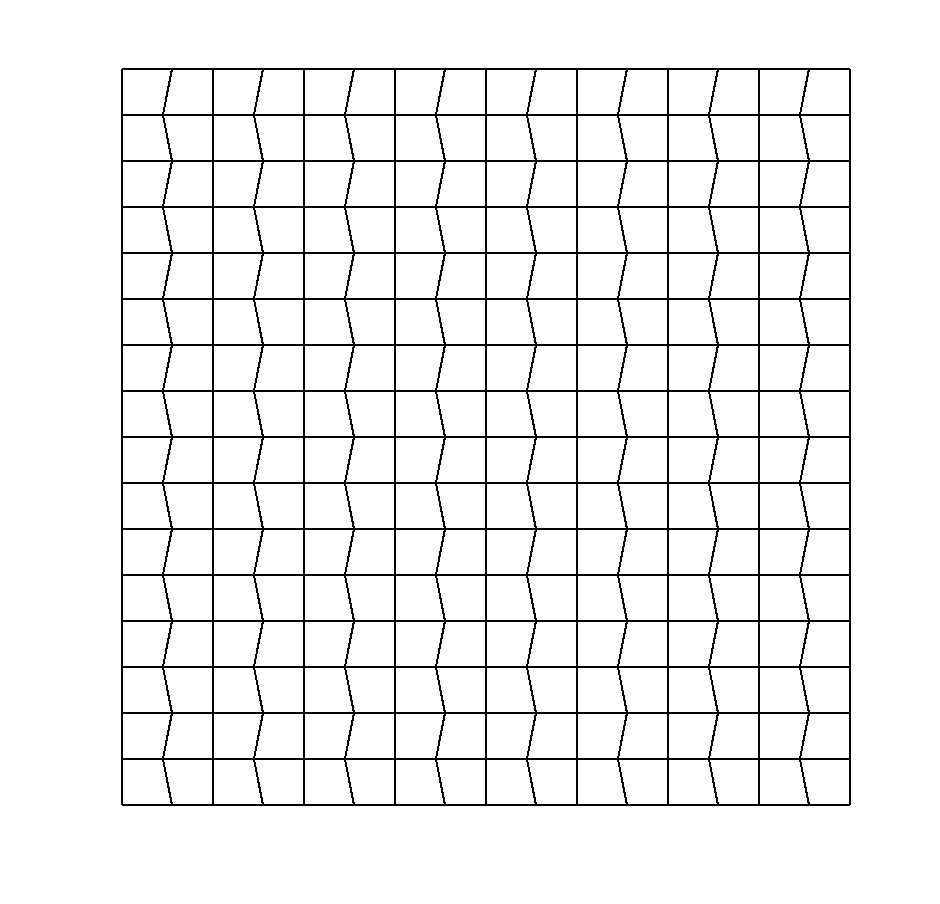}\\
\centering\includegraphics[height=4.2cm, width=4.2cm]{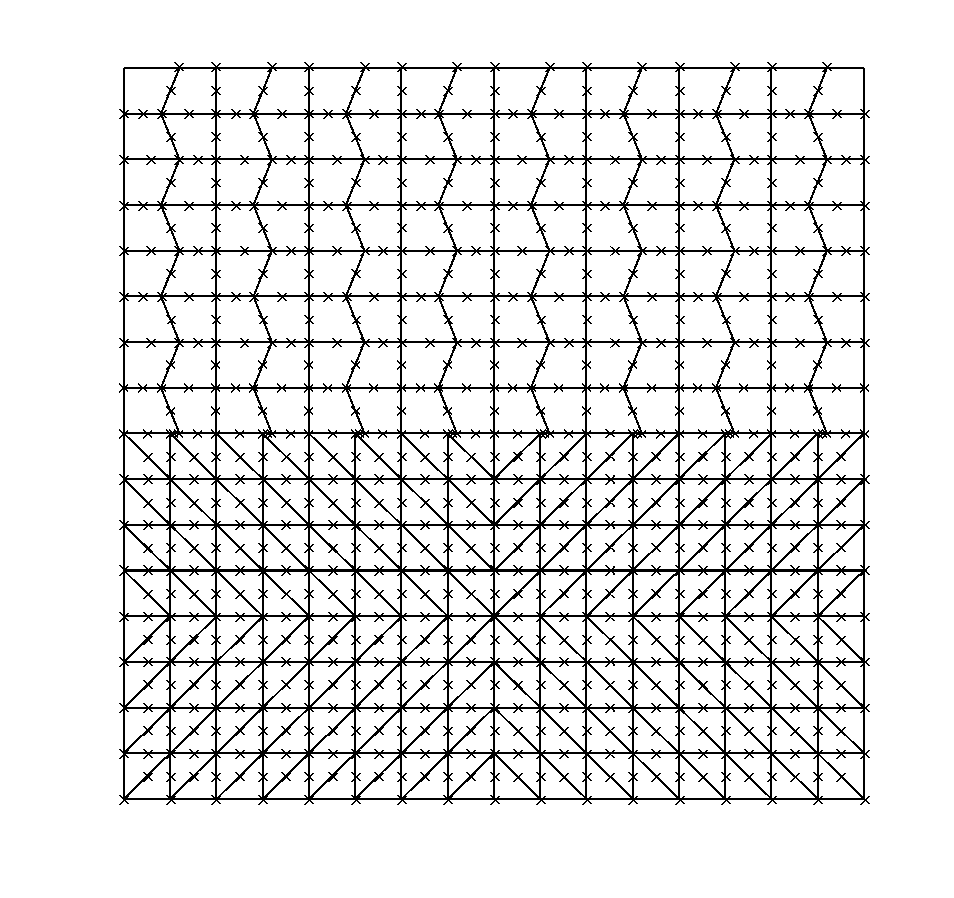}
\centering\includegraphics[height=4.2cm, width=4.2cm]{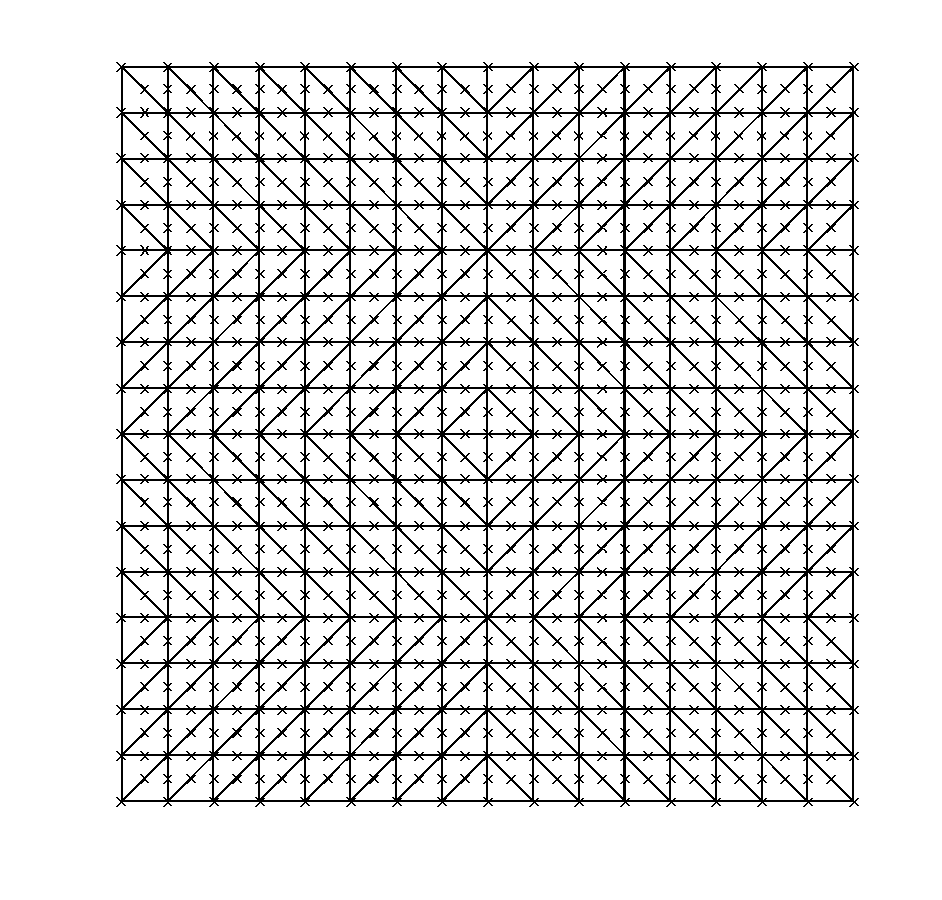}
\caption{\label{fig:mesh} Sample meshes: $\CT_h^1$ (top left), $\CT_h^2$ (top center), $\CT_h^3$ (top right), $\CT_h^4$ (bottom left), $\CT_h^5$ (bottom right).}
\end{center}
\end{figure}
%

\subsection{Test 1: Convex domain}

In this test, we consider the convex domain $\O = (-1,1)^{2}$ with zero boundary conditions, and $\beta=(1,0)^{t}$. Observe that due to the convexity of this domain, the convergence rates are the optimal as we expect. The meshes considered for this test are $\CT_h^{1}$, $\CT_h^{2}$ and $\CT_h^{3}$. In tables \ref{tabla1},\ref{tabla2} and \ref{tabla3} we present the first four approximated frequencies and their orders of convergence on these three different meshes.

\begin{table}[H]
\caption{Four lowest approximated frequencies, orders of convergence and extrapolated frequencies computed with $\mathcal{T}_{h}^{1}$.}
\label{tabla1}
\begin{center}
\resizebox{13cm}{!}{
\begin{tabular}{|c|c|c|c|c|c|c|c|} \hline
$\omega_{hi}$ & $\mathrm{N}$=16 & $\mathrm{N}$=32 & $\mathrm{N}$=64 & $\mathrm{N}$=128 & Order & Exact. \\ \hline 
$\omega_{h1}$ & 13.92216 & 13.68985 & 13.62982 & 13.61468 & 1.96 & 13.60931 \\
$\omega_{h2}$ & 24.08035 & 23.37254 & 23.19098 & 23.14515 & 1.97 & 23.12934 \\
$\omega_{h3}$ & 24.33081 & 23.64147 & 23.47925 & 23.43710 & 2.07 & 23.42628 \\
$\omega_{h4}$ & 34.88566 & 32.93985 & 32.46160 & 32.33912 & 2.02 & 32.30257 \\
\hline
\end{tabular}}
\end{center}
\end{table}

\begin{table}[H]
\caption{Four lowest approximated frequencies, orders of convergence and extrapolated frequencies  computed with $\mathcal{T}_{h}^{2}$.}
\label{tabla2}
\begin{center}
\resizebox{13cm}{!}{
\begin{tabular}{|c|c|c|c|c|c|c|c|} \hline
$\omega_{hi}$ & $\mathrm{N}$=14 & $\mathrm{N}$=30 & $\mathrm{N}$=66 & $\mathrm{N}$=121 & Order & Exact. \\ \hline 
$\omega_{h1}$ & 13.75551 & 13.63935 & 13.61705 & 13.61130 & 2.11 & 13.61056 \\
$\omega_{h2}$ & 23.61476 & 23.24318 & 23.15738 & 23.13621 & 1.90 & 23.12977 \\
$\omega_{h3}$ & 23.91363 & 23.53488 & 23.44976 & 23.42958 & 1.94 & 23.42363 \\
$\omega_{h4}$ & 33.33498 & 32.54061 & 32.35652 & 32.31206 & 1.90 & 32.29795 \\
\hline
\end{tabular}}
\end{center}
\end{table}

\begin{table}[H]
\caption{Four lowest approximated frequencies, orders of convergence and extrapolated frequencies  computed with $\mathcal{T}_{h}^{3}$.}
\label{tabla3}
\begin{center}
\resizebox{13cm}{!}{
\begin{tabular}{|c|c|c|c|c|c|c|c|} \hline
$\omega_{hi}$ & $\mathrm{N}$=16 & $\mathrm{N}$=32 & $\mathrm{N}$=64 & $\mathrm{N}$=128 & Order & Exact. \\ \hline 
$\omega_{h1}$ & 13.90229 & 13.68293 & 13.62793 & 13.61418 & 2.00 & 13.60966 \\
$\omega_{h2}$ & 24.02186 & 23.35522 & 23.18628 & 23.14389 & 1.98 & 23.12908 \\
$\omega_{h3}$ & 24.19279 & 23.61418 & 23.47066 & 23.43489 & 2.01 & 23.42314 \\
$\omega_{h4}$ & 34.67160 & 32.90207 & 32.44979 & 32.33609 & 1.97 & 32.29576 \\
\hline
\end{tabular}}
\end{center}
\end{table}

We observe from Tables \ref{tabla1}, \ref{tabla2} and \ref{tabla3} that the computed order of convergence matches with theoretical order of convergence, and independent of the mesh type.
 It is clear that the convexity of the domain together with the null Dirichlet boundary condition lead to smooth eigenfunctions and hence, the quadratic order of convergence of the method.  For sake of completeness, in Figure \ref{fig:eigs1} we present plots of first three eigenfunctions, 
magnitude of the velocity fields and the associated pressures.

\begin{figure}[h!]
\begin{center}
\centering\includegraphics[height=4.5cm, width=5.0cm]{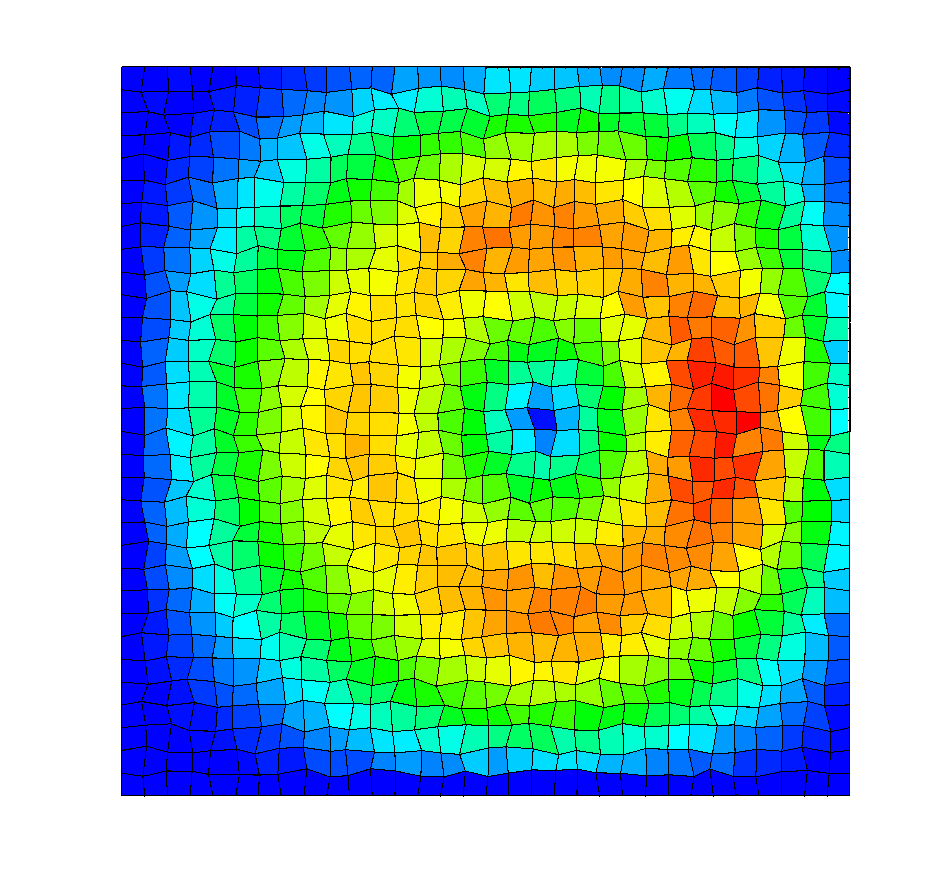} 
\centering\includegraphics[height=4.5cm, width=5.0cm]{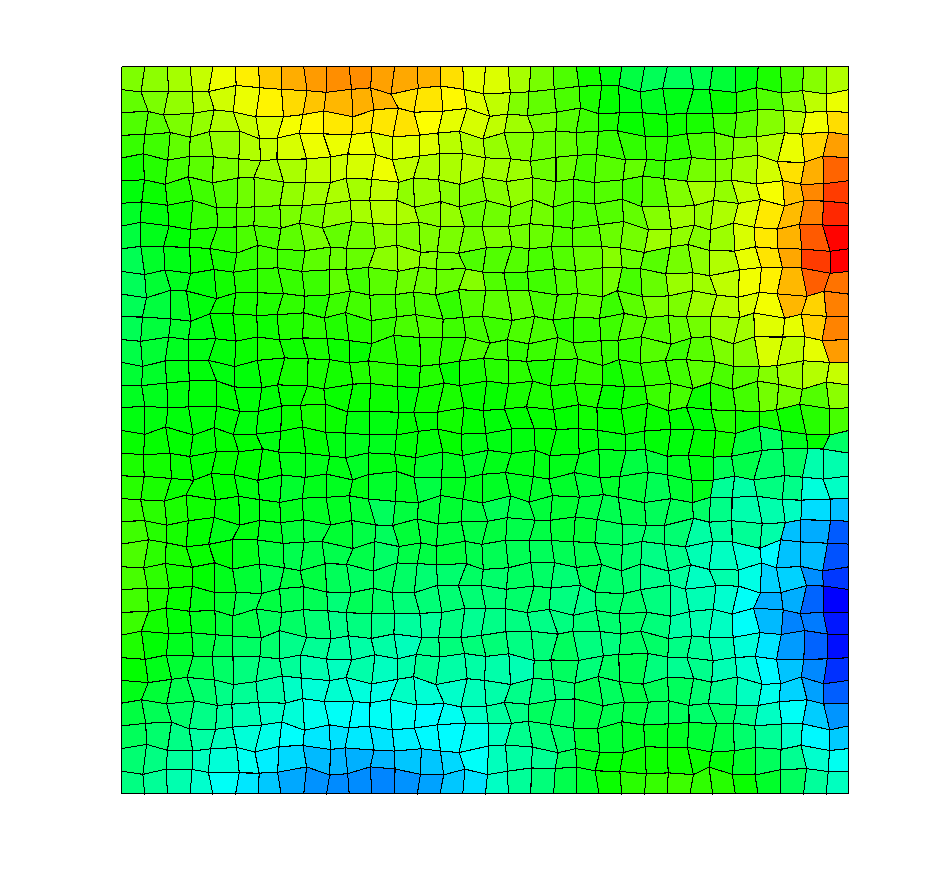}\\
\centering\includegraphics[height=4.5cm, width=5.0cm]{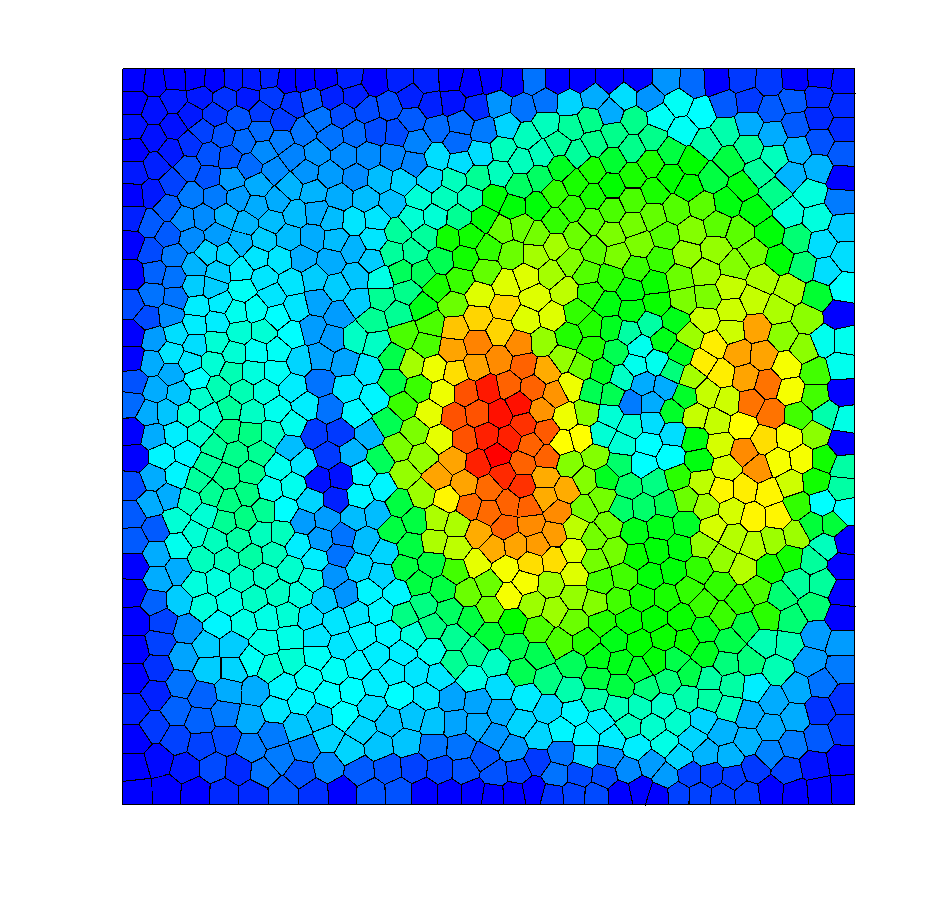}
\centering\includegraphics[height=4.5cm, width=5.0cm]{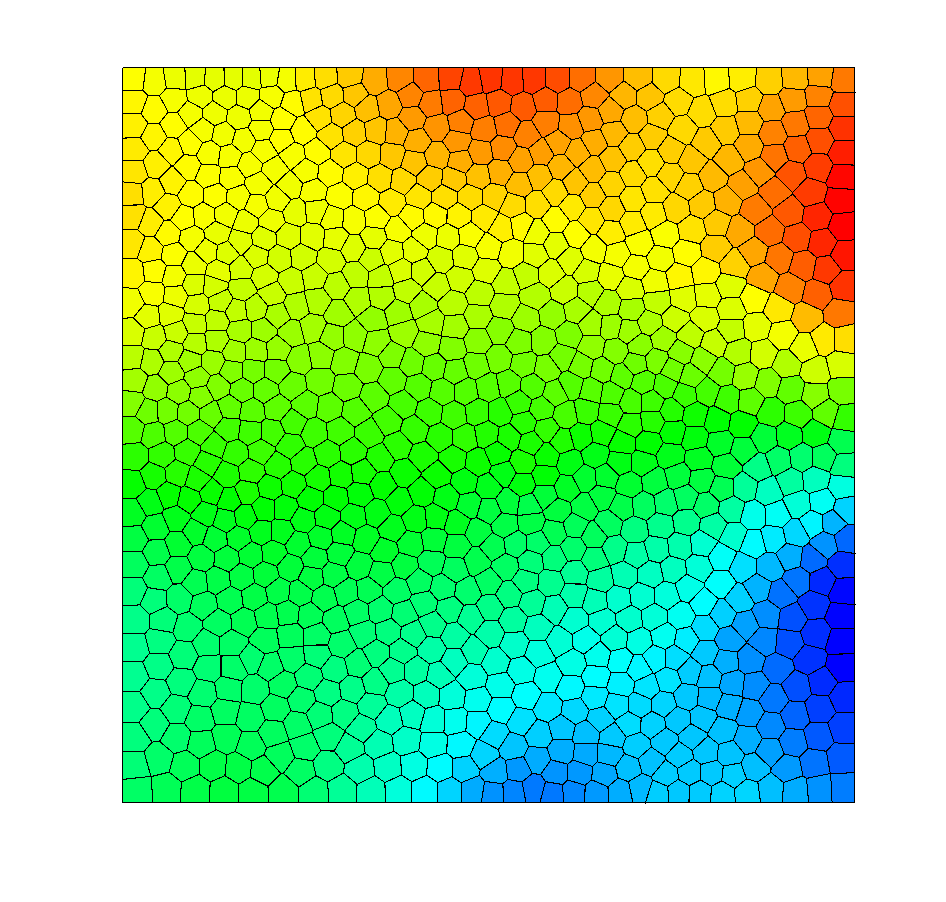}\\
\centering\includegraphics[height=4.5cm, width=5.0cm]{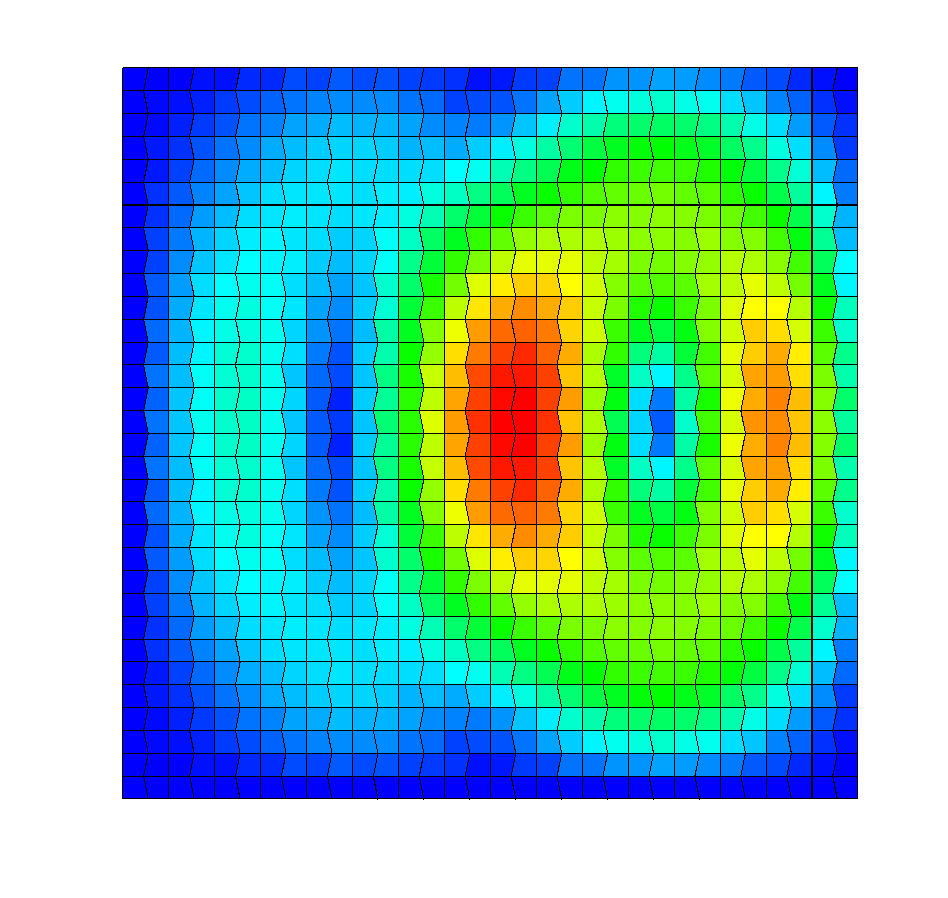}
\centering\includegraphics[height=4.5cm, width=5.0cm]{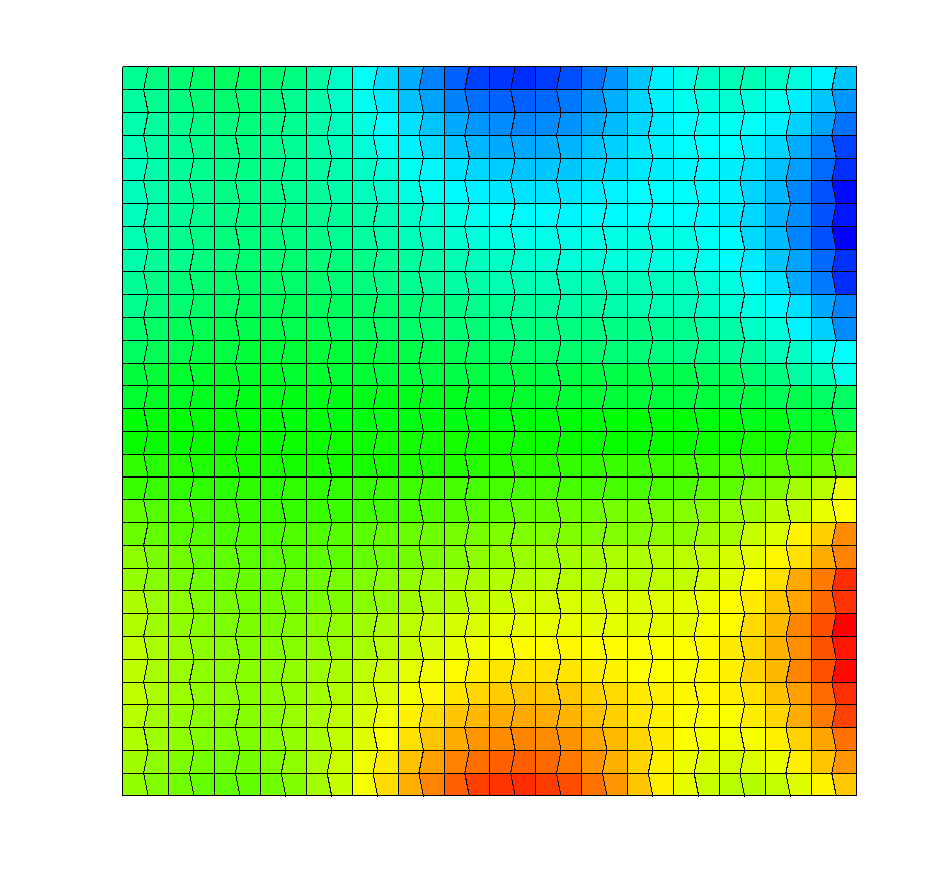}
\caption{\label{fig:eigs1} First, second and third magnitude of the eigenfunctions and their associated pressures. First column: $\bu_{h1}$, $\bu_{h2}$ and $\bu_{h3}$. Second column: $p_{h1}$, $p_{h2}$ and $p_{h3}$.}
\end{center}
\end{figure}

 \subsection{L-shaped Domain} 
The domain for this test is $\O=(-1,1)\times (-1,1)\setminus (-1,0)\times (-1,0)$
 and the only boundary condition is $\bu = \0$. It is clear that for the L-shaped domain, the re-entrant angle leads to a lack of regularity for some eigenfunctions, as has been studied in \cite{MR2473688} for the Stokes eigenvalue problem, the convergence rates of the errors for the eigenvalues, vary between $1.7 \leq r \leq2$, depending on the regularity of the corresponding eigenfunction. We expect that for the Oseen eigenvalue problem the results must be similar. For this test, we consider the mesh $\CT_h^{6}$ presented in Figure \ref{fig:Ldomain}.
 
\begin{figure}[h!]
\begin{center}
\centering\includegraphics[height=4.2cm, width=4.2cm]{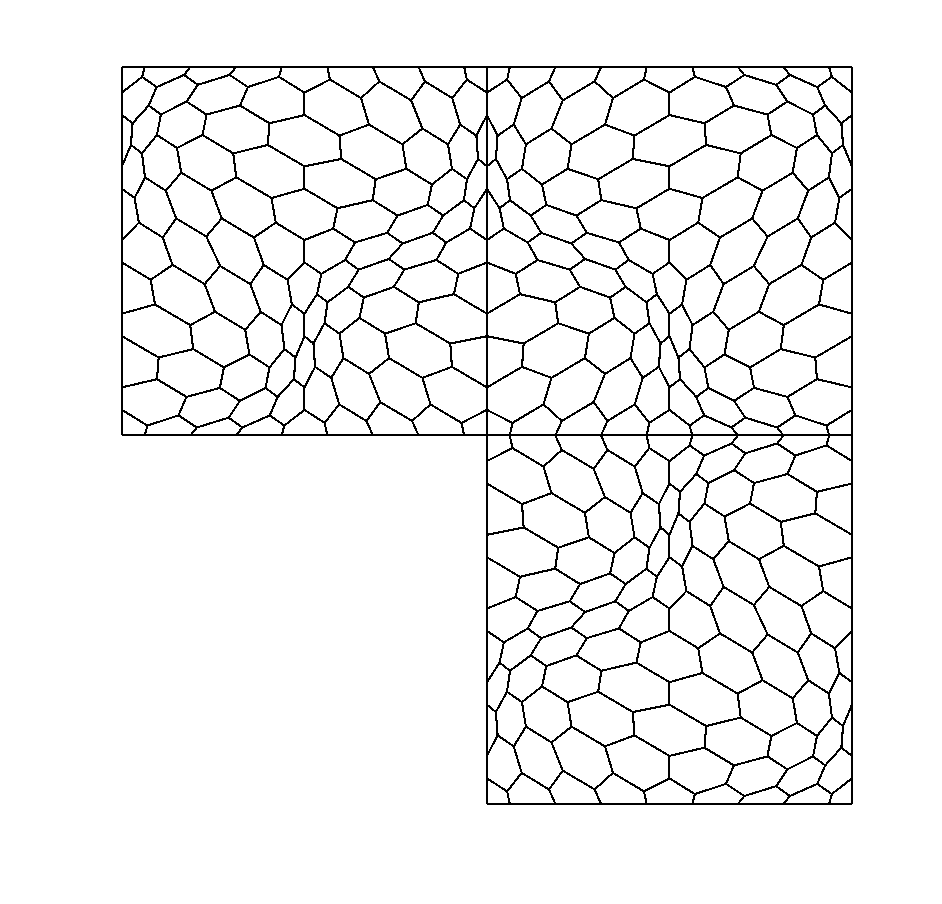} 
\caption{\label{fig:Ldomain} Sample of mesh $\CT_h^{6}$.}
\end{center}
\end{figure} 

\begin{table}[H]
\caption{Four lowest approximated frequencies, orders of convergence, and extrapolated frequencies  computed with $\mathcal{T}_{h}^{6}$.}
\label{tabla4}
\begin{center}
\resizebox{13cm}{!}{
\begin{tabular}{|c|c|c|c|c|c|c|c|} \hline
$\omega_{hi}$ & $\mathrm{N}$=38 & $\mathrm{N}$=54 & $\mathrm{N}$=70 & $\mathrm{N}$=90 & Order & Exact. \\ \hline 
$\omega_{h1}$ & 33.48850 & 33.26270 & 33.16583 & 33.10531 & 1.74 & 32.99534 \\
$\omega_{h2}$ & 37.47176 & 37.28385 & 37.21317 & 37.17461 & 2.19 & 37.12148 \\
$\omega_{h3}$ & 42.80303 & 42.59293 & 42.51336 & 42.47086 & 2.20 & 42.41193 \\
$\omega_{h4}$ & 49.79904 & 49.51086 & 49.40294 & 49.34772 & 2.26 & 49.27180 \\
\hline
\end{tabular}}
\end{center}
\end{table}

As we expect, for the first eigenvalue we attained a convergence order $1.74$ since the associated eigenfunction is singular where the re-entrant angle is located, whereas the rest of the computed eigenvalues converge with order two since the associated eigenfunctions are smooth enough. This test also confirms that the regularity of the eigenfunctions on non-convex domains behave as the Stokes eigenfunctions. In Figure \ref{fig_Lshapes} we present plots of the magnitude of the velocity and the pressure fluctuation for the first three eigenfunctions. 
\begin{figure}[h!]
\label{fig_Lshapes}
\begin{center}
\centering\includegraphics[height=4.5cm, width=5.0cm]{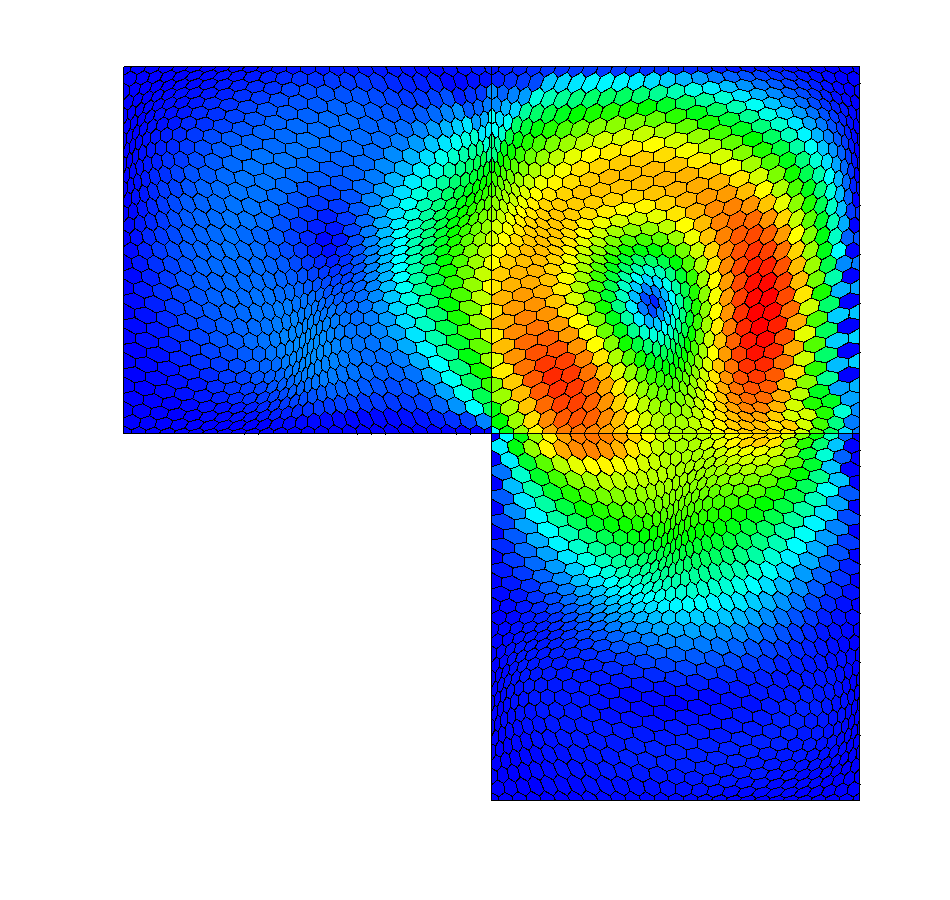} 
\centering\includegraphics[height=4.5cm, width=5.0cm]{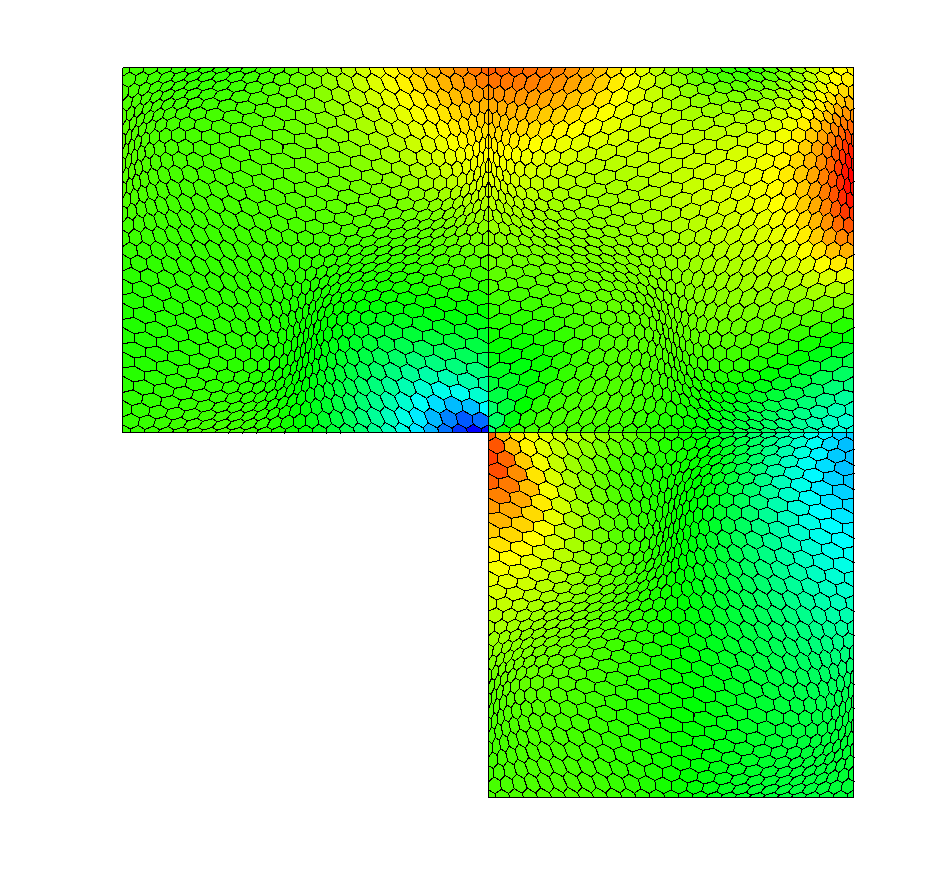} \\
\centering\includegraphics[height=4.5cm, width=5.0cm]{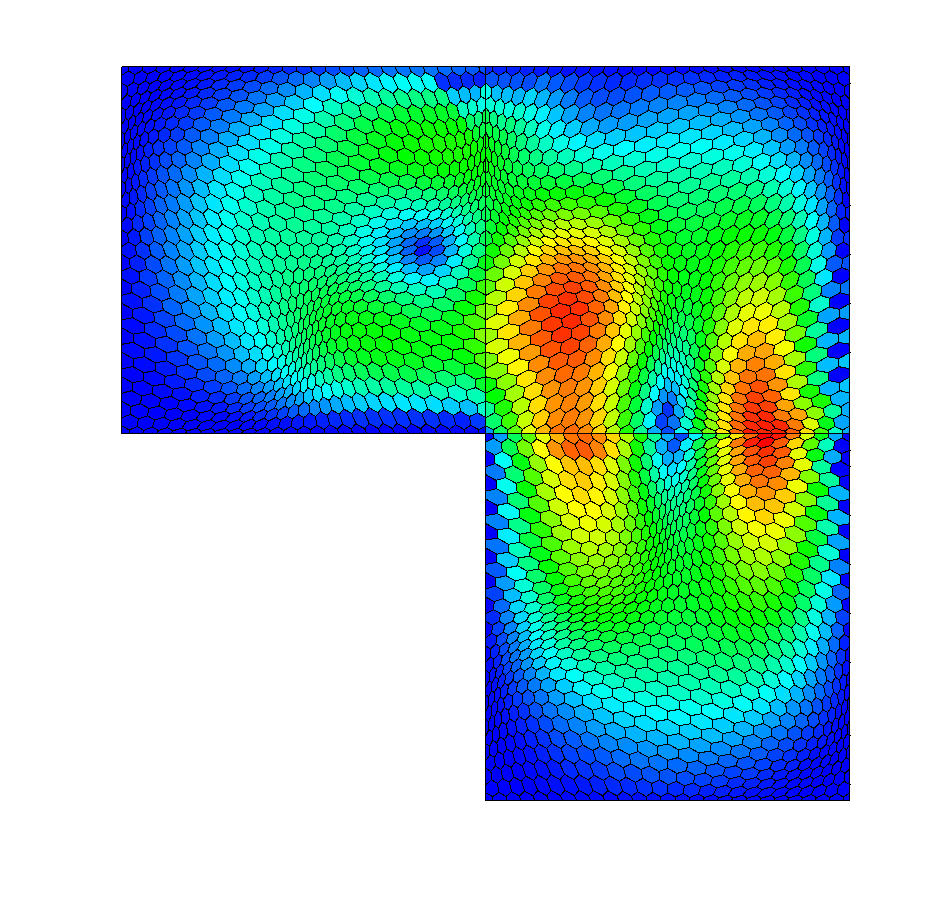} 
\centering\includegraphics[height=4.5cm, width=5.0cm]{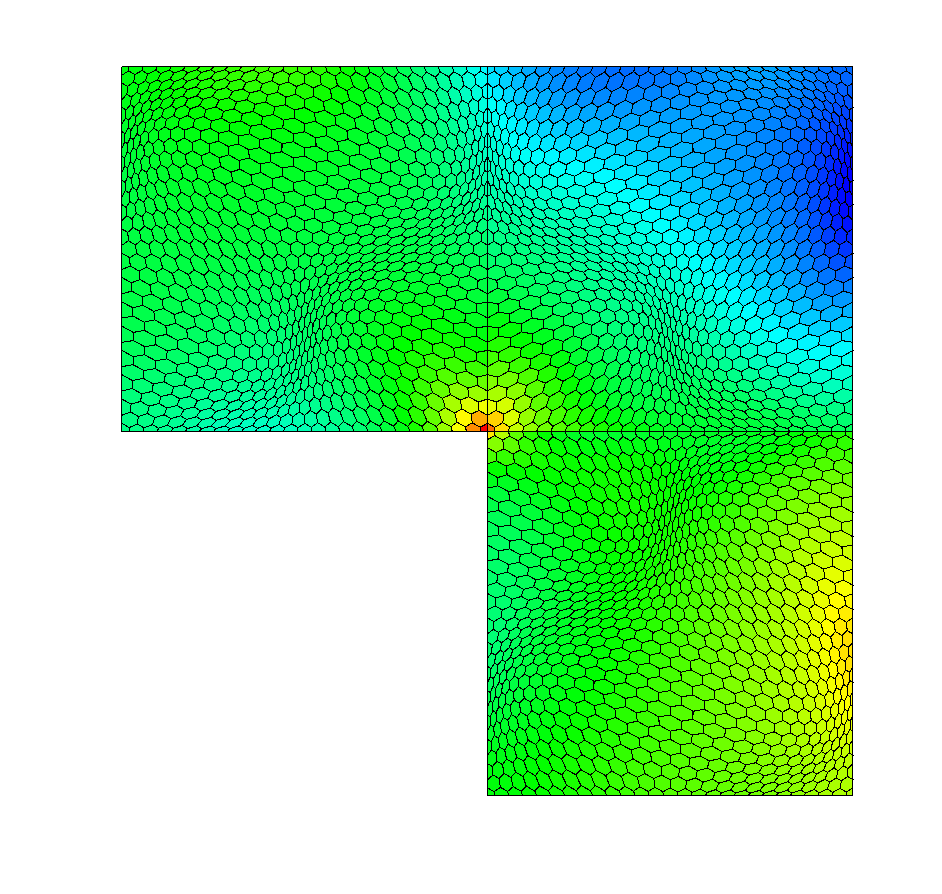} \\
\centering\includegraphics[height=4.5cm, width=5.0cm]{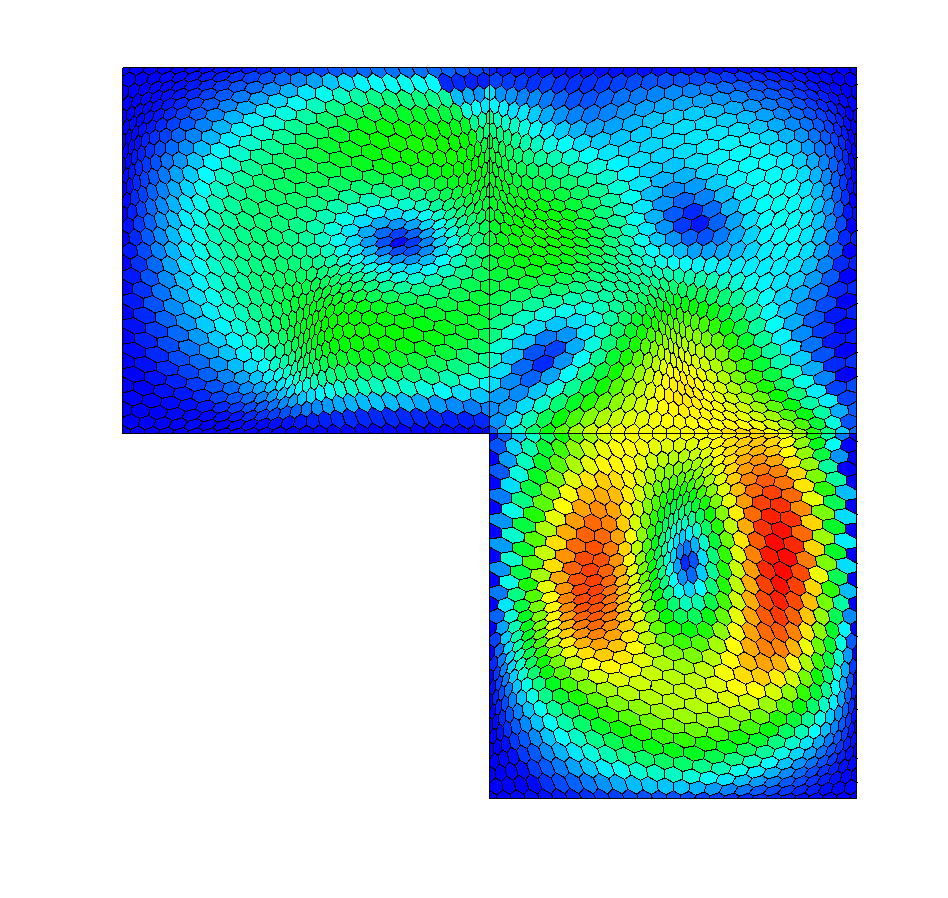} 
\centering\includegraphics[height=4.5cm, width=5.0cm]{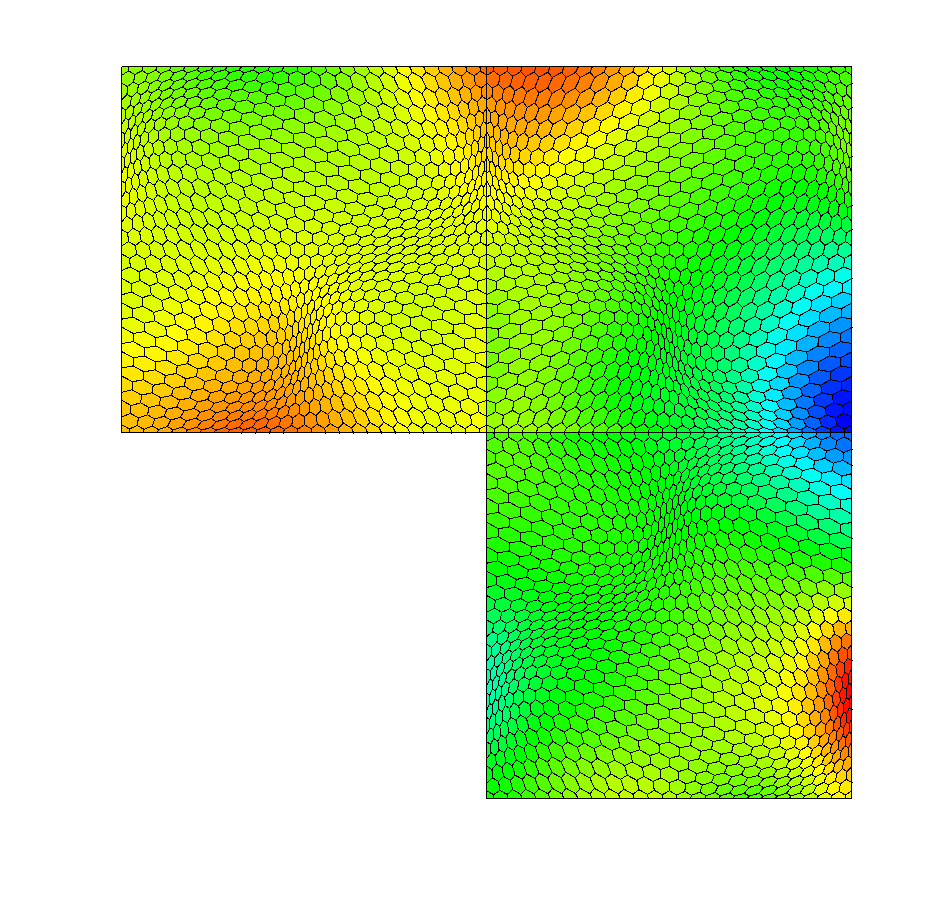}
\caption{\label{fig:eigs2} First, second and third magnitude of the eigenfunctions and their associated pressures for the L-shaped domain. First column: $\bu_{h1}$, $\bu_{h2}$ and $\bu_{h3}$. Second column: $p_{h1}$, $p_{h2}$ and $p_{h3}$.}
\end{center}
\end{figure}

\subsection{Effects of the stabilization in the computed spectrum}
The goal of this test is to observe the performance of the method with respect to the stabilization parameters, inherent on the VEM, in the context of the appearance of possible spurious eigenvalues. Let us remark that this type of analysis has been also considered on other problems and methods that depend on stabilizations,  such as \cite{adak2022vem,MR3962898,Lepe2021,MR4050542}. For this experiment, we consider the following problem: Given $\Omega \subset \mathbb{R}^{2}$ with $\partial\Omega = \Gamma_N\,\cup\,\Gamma_D$, where $|\Gamma_D|>0$, find the velocity $\bu$ and pressure $p$ such that
\begin{equation}\label{def:oseen-eigenvalue_mixto}
\left\{
\begin{array}{rcll}
-\nu\Delta \bu + (\boldsymbol{\beta}\cdot\nabla)\bu + \nabla p&=&\lambda\bu,&\text{in}\,\O,\\
\div \bu&=&0,&\text{in}\,\O,\\
\bu &=&\boldsymbol{0},&\text{on}\,\partial\O,\\
(\nabla \bu - p\boldsymbol{I})\boldsymbol{n} &=&\textbf{0}, &\text{in}\,\partial\O,
\end{array}
\right.
\end{equation}
where $\boldsymbol{I} \in \mathbb{R}^{2\times 2}$ is the identity tensor. Observe that in this case, spurious eigenvalues can be introduced for different choices of stabilization parameter. For this test, we consider $\Omega = (0,1)^{2}$ and $\boldsymbol{\beta}=(1,0)^{t}$.
Let us remark that depending on the configuration of the problem, namely physical parameters, boundary conditions, nature of the meshes, etc., the spurious may appear or not.

In the forthcoming tables we report the computed spectrum with different values for the stabilization parameter. To make matters precise, we have used the standard dofi-dofi stabilization scaled with a parameter $\alpha_E$. The numbers inside boxes represent spurious eigenvalues which we identify by looking the associated eigenfunction which in this case, does not represent a physical vibration mode. To perform the tests we have considered two polygonal meshes ($\mathcal{T}_{h}^{4}$ and $\mathcal{T}_{h}^{5}$) and we have fixed the refinement level on $\mathrm{N}=8$. 

\begin{table}[H]
\label{tablaespureos3}
\caption{Computed frequencies for $\mathcal{T}_{h}^{4}$ and $\mathrm{N}= 8$, for different stabilization parameters.}
\begin{center}
\resizebox{13cm}{!}{
\begin{tabular}{|c|c|c|c|c|c|c|} \hline
$\alpha_{E}=1/32$ &$\alpha_{E}= 1/16$ & $\alpha_{E}=1/4$ & $\alpha_{E}=1$ &$\alpha_{E}= 4$ &$ \alpha_{E}=16$ & $\alpha_{E}=32$  \\ 
\cline{1-7}
2.07098 & 2.29713 & 2.54288 & 2.70590 & 2.85551 & 2.96811 & 3.01781 \\
3.02694 & 3.73371 & 5.36000 & 6.46190 & 7.10930 & 7.73505 & 8.21554\\
\fbox{7.92036} & 10.59601 & 13.67759 & 15.40457 & 16.55426 & 18.50199 & 20.51561 \\
 \fbox{8.43356} & \fbox{12.85456} & 20.23071 & 23.52387 & 25.35965 & 26.61102 & 27.24423 \\
\fbox{8.49612} & 16.48511 & 22.26127 & 27.02276 & 30.97145 & 37.83053 & 44.48394 \\
\fbox{8.53684} & \fbox{16.86904} & 32.11931 & 41.08455 & 46.51663 & 53.71089 & 60.55640 \\
\fbox{8.58618} & \fbox{16.99124} & 37.63607 & 49.27221 & 57.94917 & 76.18968 & 80.24629 \\
\fbox{8.63582} & \fbox{17.08513} & 42.46586 & 59.66181 & 69.01996 & 78.28127 & 100.29973 \\
\fbox{8.68978} & \fbox{17.24043} & \fbox{44.68864} & 63.60100 & 82.35945 & 108.75566 & 130.02175 \\
\fbox{8.70030} & \fbox{17.29395} & \fbox{50.94633} & 74.07543 & 94.21901 & 120.83558 & 136.36453 \\ \hline
\end{tabular}}
\end{center}
\end{table}

\begin{table}[H]
\label{tablaespureos4}
\caption{Computed frequencies for $\mathcal{T}_{h}^{5}$ and $\mathrm{N}= 8$, with different stabilization parameters.}
\begin{center}
\resizebox{13cm}{!}{
\begin{tabular}{|c|c|c|c|c|c|c|} \hline
$\alpha_{E}=1/32$ & $\alpha_{E}=1/16$ & $\alpha_{E}=1/4$ & $\alpha_E=1$ & $\alpha_E=4$ & $\alpha_E=16$ & $\alpha_E=32$ \\ 
\cline{1-7}
 2.48684 & 2.51670 & 2.59365 & 2.71563 & 2.85707 & 2.96714 & 3.01578 \\
3.11397 & 3.77960 & 5.37074 & 6.46026 & 7.09728 & 7.70639 & 8.18150 \\
 \fbox{8.31185} & 12.25624 & 14.26276 & 15.52922 & 16.47288 & 17.78163 & 18.83300 \\
\fbox{8.34868} & \fbox{16.52221} & \fbox{22.36009+0.31349i} & 23.90778 & 25.46364 & 26.66147 & 27.27722 \\
 \fbox{8.42230} & \fbox{16.61311} & \fbox{22.36009-0.31349i} & 27.11003 & 30.25017 & 34.02411 & 37.26804 \\
\fbox{8.43810} & \fbox{16.82133} & 37.69737 & 42.53760 & 45.77237 & 48.96153 & 50.94096 \\
\fbox{8.44880} & \fbox{16.83215} & 39.96292 & 49.27091 & 55.47536 & 64.67996 & 72.34620 \\
 \fbox{8.48054} & \fbox{16.83828} & 46.82712 & 61.77897 & 68.92870 & 74.21710 & 77.21338 \\
\fbox{8.51422} & \fbox{16.93777} & 54.76065 & 66.40818 & 79.81304 & 91.89810 & 99.36315 \\
\fbox{8.52497} & \fbox{16.99193} & 59.18041 & 74.79963 & 90.37610 & 106.71571 & 110.87567 \\ \hline
\end{tabular}}
\end{center}
\end{table}

From Tables \ref{tablaespureos3} and \ref{tablaespureos4} we observe that spurious eigenvalues arises when $\alpha_{E}<1$ for $\CT_h^4$ and $\CT_h^5$. Moreover, the spurious eigenvalues begin to vanish from the spectrum when $\alpha_E>1$. This type of behavior is the expected according to the literature (see \cite{adak2022vem,Lepe2021} for instance). We observe that when $\CT_h^5$ is considered, complex eigenvalues are computed. This gives us the hint that the geometry of the meshes have an influence on the computation of the spectrum.

Now the natural question is related to the relation between the spurious eigenvalues and the refinement of the meshes. To analyze this, the strategy is to consider a certain parameter $\alpha_E$ where spurious eigenvalues appear and start to refine the mesh in order to observe the behavior of these pollution on the spectrum. For this, we consider $\alpha_E=1/16$ and meshes $\CT_h^4$ and $\CT_h^5$.  
\begin{table}[H]
\begin{center}
\caption{\label{tab21} First ten approximated eigenvalues for $\mathcal{T}_{h}^{4}$ and $\mathcal{T}_{h}^{5}$ for $\alpha_E = 1/16$.}
\resizebox{13cm}{!}{\begin{tabular}{|c|c|c|c|c|c|c|c|c|c|} \hline
& \multicolumn{4}{ |c| }{$\mathcal{T}{h}^{1}$} &\multicolumn{4}{ |c| }{$\mathcal{T}{h}^{2}$}\\ \hline
$\l_{ih}$ & $N = 8$ & $N = 16$ & $N = 32$ & $N = 64$ &$N = 8$ & $N = 16$ & $N = 32$ & $N = 64$ \\ \hline \hline 
 $\l_{1h}$ & 2.29713  &  2.44138  &  2.46927    &2.47193 &  2.51670 & 2.49773 & 2.48329 & 2.47541 \\
 $\l_{2h}$ &  3.73371 &   4.53255   & 5.20913    &5.67485 &3.77960 & 4.54705 & 5.21418 & 5.67652 \\
 $\l_{3h}$ &  10.59601&    12.96465  & 14.26285  & 14.89066& 12.25624 & 13.63713 & 14.46537 & 14.94452 \\
 $\l_{4h}$ &  \fbox{12.85456}  & 19.83441 &  21.70281 &  22.11427&\fbox{16.52221} & 20.67849 & 22.28230 & 22.25192 \\
 $\l_{5h}$ & 16.48511 &    20.16112 &  23.06503  & 25.11909& \fbox{16.61311}&  21.88542 & 23.18495 & 25.17128 \\
 $\l_{6h}$ & \fbox{16.86904}  & 31.90500  & 38.92423  & 41.95938& \fbox{16.82133}&  \fbox{35.73238} & 40.49989 & 42.70875 \\
 $\l_{7h}$ &  \fbox{16.99124} &   \fbox{34.33244} &  41.60708 &  45.72429&\fbox{16.83215}&  40.56047 & 43.42275 & 45.87948 \\
 $\l_{8h}$ &  \fbox{17.08513} &   \fbox{41.27676} &  \fbox{48.66486} &  56.27042 &  \fbox{16.83828}&  \fbox{45.00339} & \fbox{49.64642} &  56.82759 \\
 $\l_{9h}$ &   \fbox{17.24043}   & \fbox{42.47989} &  55.42921  & 60.05007 & \fbox{16.93777}&  58.25960 & 61.05236 &61.54236 \\
 $\l_{10h}$ &  \fbox{17.29395} &  \fbox{49.97204}&   65.69181 &  71.00460&   \fbox{16.99193} & 59.26476 & 66.82080 &71.42255\\ \hline
\end{tabular}}
\end{center}
\end{table}

From Table \ref{tab21} we observe  that the pollution on the computed spectrum begins to vanish as the meshes are refined, as we expect. Finally, in Figure \ref{figesp} we present the plots of the first four eigenfunctions associated to system \eqref{def:oseen-eigenvalue_mixto}. 

\begin{figure}[h!]
\begin{center}
\centering\includegraphics[height=4.5cm, width=5.0cm]{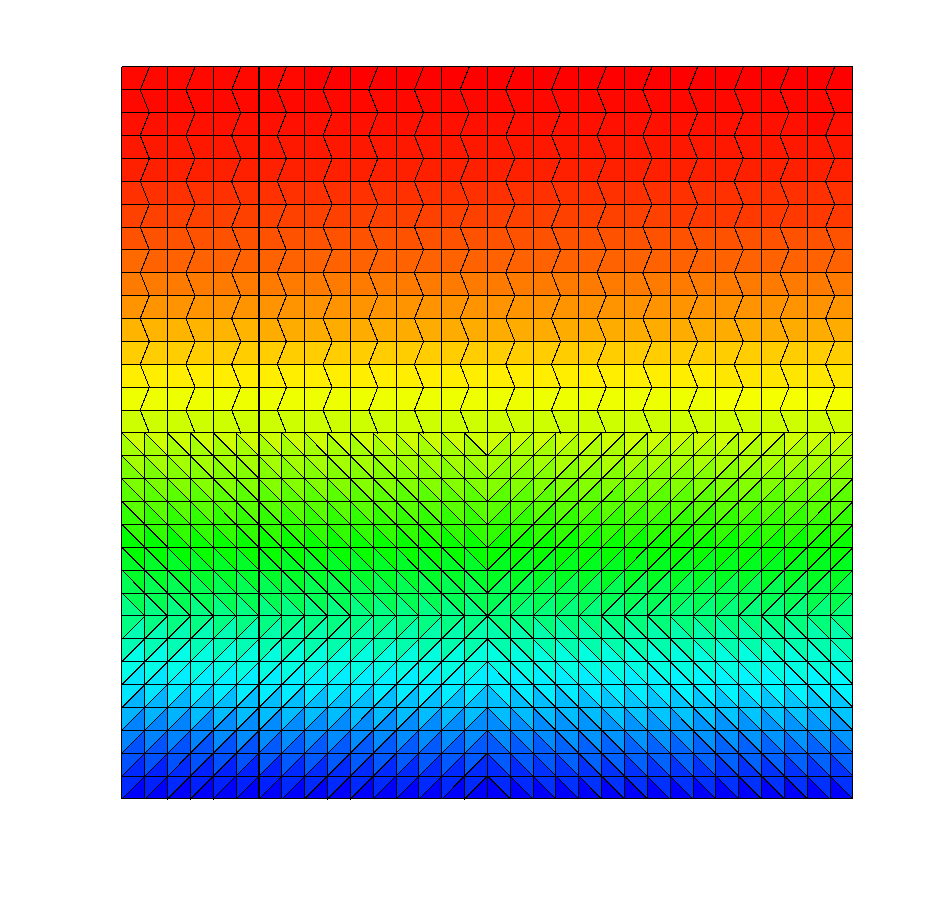} 
\centering\includegraphics[height=4.5cm, width=5.0cm]{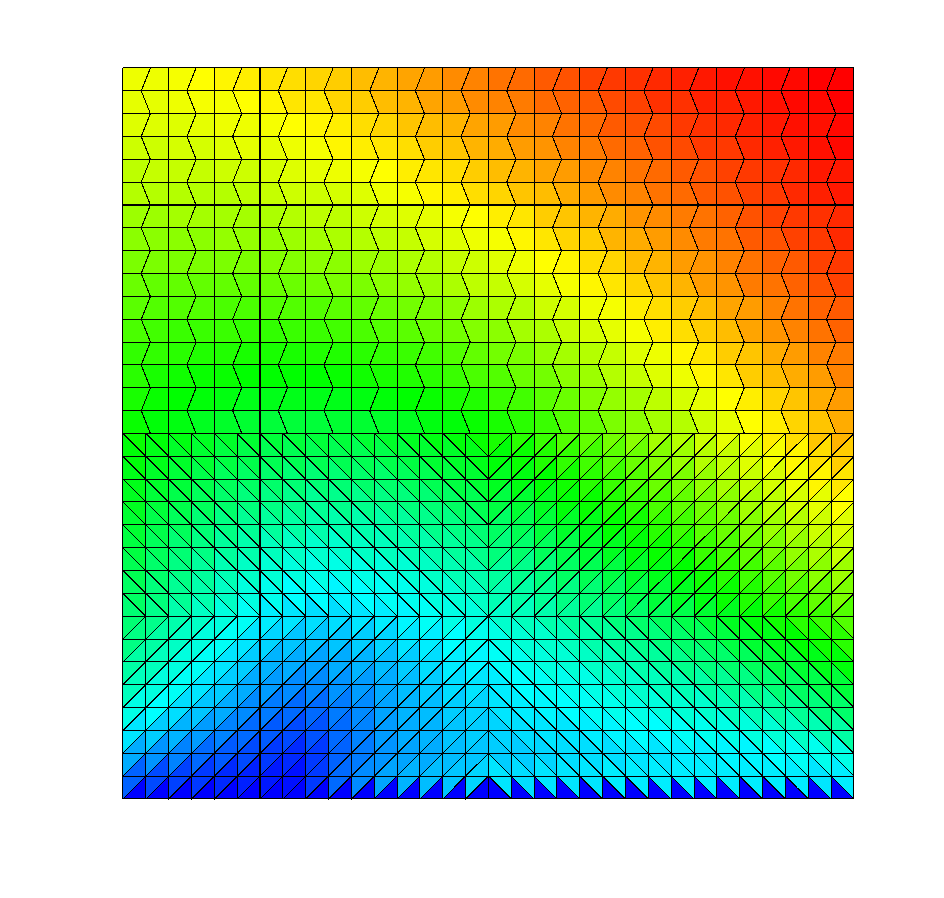}\\
\centering\includegraphics[height=4.5cm, width=5.0cm]{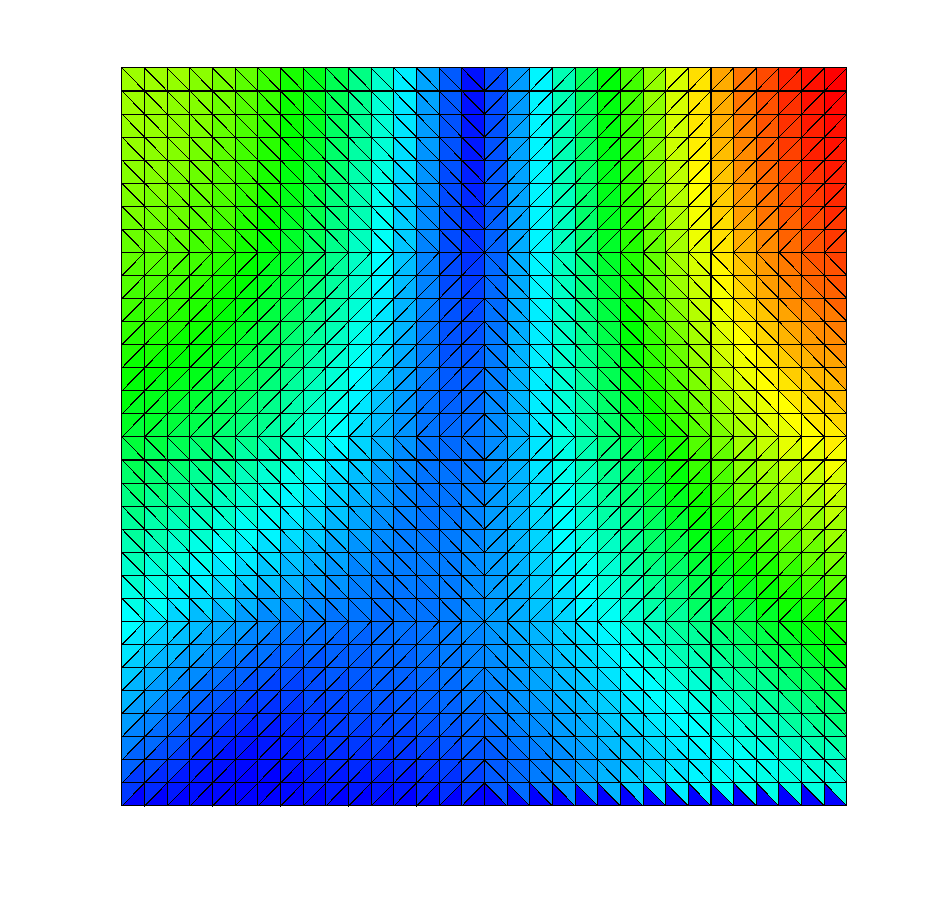}
\centering\includegraphics[height=4.5cm, width=5.0cm]{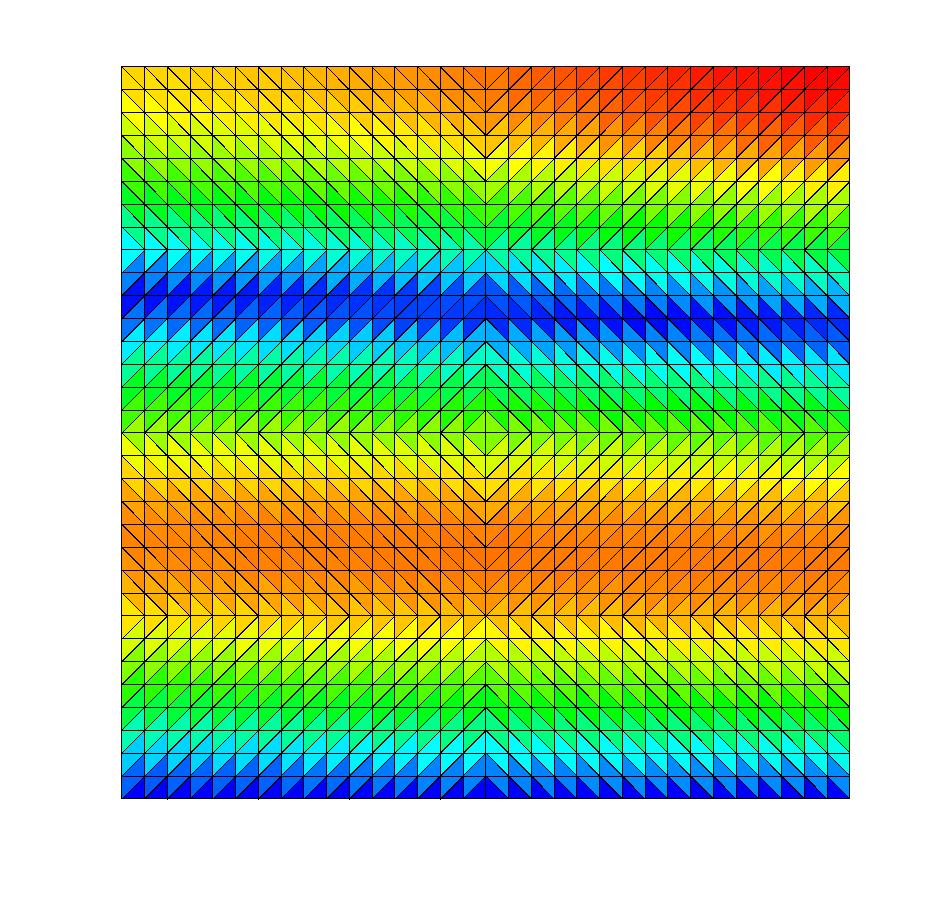}\\
\caption{\label{figesp} Magnitude of the first four eigenfunctions. From left to right: $\bu_{h1}$ and $\bu_{h2}$ (top); $\bu_{h3}$ and $\bu_{h4}$ (bottom).}
\end{center}
\end{figure}

\bibliographystyle{siamplain}
\bibliography{oseen-eigenvalue}
\end{document}